\documentclass[12pt,a4paper,reqno]{amsart}
\usepackage[margin=3.07cm,tmargin=2.0cm]{geometry}
\numberwithin{equation}{section}
\usepackage{eucal}
\usepackage{xspace}
\newcommand{\sode}{\textsc{sode}\xspace}
\let\bs\boldsymbol
\newcommand{\cala}{\mathcal{A}}
\newcommand{\calb}{\mathcal{B}}
\newcommand{\cald}{\mathcal{D}}
\newcommand{\calf}{\mathcal{F}}
\newcommand{\call}{\mathcal{L}}
\newcommand{\calt}{\CMcal{T}}
\newcommand{\calu}{\mathcal{U}}
\newcommand{\calv}{\mathcal{U}}
\newcommand{\Real}{\mathbb{R}}
\newcommand{\set}[2]{\left\{\,#1\left.\vphantom{#1#2}\,\right\vert\,#2\,\right\}}
\newcommand{\map}[3]{#1\colon#2\rightarrow#3}
\newcommand{\id}{{\operatorname{id}}}
\newcommand{\pr}{\operatorname{pr}}
\newcommand{\pai}[2]{\langle{#1},{#2}\rangle} 
\renewcommand{\Im}{\operatorname{Im}}
\newcommand{\Ker}{\operatorname{Ker}}
\newcommand{\cinfty}[1]{C^{\scriptscriptstyle\infty}(#1)}   
\renewcommand{\sec}[1]{\operatorname{Sec}(#1)} 
\newcommand{\vectorfields}[2][]{\mathfrak{X}_{#1}(#2)}
\newcommand{\pd}[2]{\frac{\partial#1}{\partial#2}} 
\newcommand{\at}[1]{\Big\vert_{#1}}
\newcommand{\cansec}{\mathbb{I}}
\def\seq 0->#1-#2->#3-#4->#5->0{0\longrightarrow
            #1\maparrow{#2}#3\maparrow{#4}#5\longrightarrow 0}
\newcommand{\maparrow}[1]{\mathrel{\mathop{\longrightarrow}\limits^{#1}}}
\newcommand{\Hor}{\operatorname{Hor}}
\newcommand{\Ver}{\operatorname{Ver}}
\def\curv{\mathrm{Curv}}
\def\Rie{\mathrm{Rie}}

\def\V{{\scriptscriptstyle\mathsf{V}}}
\def\H{{\scriptscriptstyle\mathsf{H}}}

\newcommand{\hlift}{\xi^\H}
\newcommand{\vlift}{\xi^\V}
\newcommand{\Ti}{\mathbb{T}}

\usepackage{amsthm}

\newtheoremstyle{slanted}{3pt}{3pt}{\slshape}{}{\scshape}{:}{ }{}
\newtheoremstyle{nonslanted}{3pt}{3pt}{\upshape}{}{\scshape}{:}{ }{}

\theoremstyle{slanted} 
  \newtheorem{guia}{}[section]
  \newtheorem{definition}[guia]{Definition}
  \newtheorem{theorem}[guia]{Theorem}
  \newtheorem{proposition}[guia]{Proposition}

\theoremstyle{nonslanted} 
  
  \newtheorem{remarkth}[guia]{Remark}
  
\newenvironment{remark}
  {\begin{remarkth}}
  {\hfill$\diamond$\end{remarkth}}

\renewenvironment{proof}
  {\par\normalfont
   \topsep 6pt plus 6pt
   \noindent{\scshape Proof.}\space\ignorespaces}
  {\qed\endtrivlist\medskip}

\begin{document}

\title[Linearization of a nonlinear connection]{Linearization of nonlinear connections\\ on vector and affine bundles,\\ and some applications}

\author[E.\ Mart\'{\i}nez]{Eduardo Mart\'{\i}nez}
\address{Eduardo Mart\'{\i}nez:
IUMA and Departamento de Matem\'atica Aplicada,
Universidad de Zaragoza,
Pedro Cerbuna 12,
50009 Zaragoza, Spain}
\email{emf@unizar.es}
\thanks{I would like to thank to the anonymous referees for a careful reading of this manuscript and for their comments and suggestions. Partial financial support from MINECO (Spain) grant MTM2015-64166-C2-1-P and from Gobierno de Arag\'on (Spain) grant DGA-E24/1 is acknowledged}



\begin{abstract} 
A linear connection is associated to a nonlinear connection on a vector bundle by a linearization procedure. Our definition is intrinsic in terms of vector fields on the bundle. For a connection on an affine bundle our procedure can be applied after homogenization and restriction. Several applications in Classical Mechanics are provided.   
\end{abstract}

\maketitle

\section{Introduction}
\label{introduction}

Connection theory is an important subject in Differential Geometry with many relevant applications in different areas of Physics~\cite{connections.Sarda}. In Classical Mechanics, linear and nonlinear connections appear as a valuable tool in the study of several problems, as reduction by symmetry groups~\cite{connections.Marsden}, some special non-holonomic systems~\cite{Marsden.nonholonomic.horizontal, Cardin.Favretti} or the geometry of second-order differential equations and the Inverse Problem~\cite{IP.nonautonomous}, to mention a few. In particular, the geometry of second-order differential equations (\sode) is dominated by the existence of a nonlinear connection canonically associated to such differential equation on the tangent bundle for the autonomous case~\cite{Crampin.nonlinear.connection} and on a first jet bundle for the non-autonomous case~\cite{Crampin.connection.non-autonomous}.

A second step in this direction was the introduction of a linear connection associated to a nonlinear one by a kind of linearization procedure~\cite{linearizable.SODE}. It was successfully used in the problem of the existence and construction of coordinates in which a \sode is linear in velocities or linear in all coordinates. In fact, this connection was implicit in our previous work on derivations of forms~\cite{deriv, deriv2} and it was the key in the solution of the problem of decoupling a \sode~\cite{decoupling}. Later the linearization of the connection was extended to the non-autonomous case~\cite{Bianchi}, where the main difficulty was caused by the fact that the first jet bundle in not a vector bundle but an affine bundle. 

To our knowledge, the linearization of a nonlinear connection on a vector bundle appears for the first time in~\cite{Vilms}. Vilms refers to such connection as the Berwald connection, since it was apparently used by Berwald in Finsler Geometry, where the tangent bundle is the relevant bundle. However, the definition provided by Vilms is given in local coordinates, what makes unclear its meaning, and more importantly it makes difficult to extract relevant information about the original connection. In this direction our definition in~\cite{linearizable.SODE,tesis} was given intrinsically in terms of brackets of vector fields on the tangent bundle. In this paper we will extend our definition to cover the case of a nonlinear connection on a general vector bundle.

\medskip 

The aim of this paper is twofold. On one hand we will provide an intrinsic definition for the linearization of a nonlinear connection on a general vector bundle, not necessarily a tangent bundle. We will provide a simple expression for the covariant derivative in terms of vector fields on the total space of the bundle. This generalization allows us to extend some results known to be true in the case of the tangent bundle in order to study properties of nonlinear connections defined on other vector bundles. The case of a nonlinear connection defined on a cotangent bundle (the natural space to study Hamiltonian Mechanics) will be considered in some detail. A simple application will show how our theory can be used to determine interesting properties of completely integrable systems, and this is just an indication of the kind of results that future research on this field will produce. 

On the other hand, we will show how to linearize a nonlinear connection defined on a general affine bundle. Our approach consists in applying the same linearization construction to a homogeneous connection defined naturally (on an appropriate vector bundle) in terms of the original nonlinear connection. This warranties that the properties of the linearized connection holds true also in the case of affine bundles. The resulting linearized connection is an affine connection (in the terminology of~\cite{con.affine.MeSaMa,con.affine.MeSa}, not to be confused with an affine connection on a manifold) and generalizes to arbitrary affine bundles the linear connection that we obtained in the case of the first jet bundle in~\cite{Bianchi,con.affine.MeSaMa}. In particular, it will be shown that the choice of the connection that we did in~\cite{Bianchi} is the natural one.

In the light of our new results, we will revisit some of the applications mentioned above about geometric properties of second-order (autonomous and non-autonomous) differential equations. On the other hand, we will give simple and shorter proofs of the most relevant results. We will consider in detail the case of a nonlinear connection on the cotangent bundle. Its main properties and the relation to the canonical symplectic structure will be established. As an application we will show how our theory can be used to extract important information in the theory of completely integrable systems, where the nonlinear connection is defined by the completely integrable structure itself. 

\medskip
\paragraph{Organization of the paper}
In Section~\ref{preliminaries} we establish the necessary preliminaries on vector bundles and Ehresmann connections  that we will be using in the rest of the paper. In Section~\ref{linearization} the linearization of a nonlinear connection on a vector bundle will be defined, and we will establish its main properties (parallel transport, curvature and homogeneity). In section~\ref{linearization.affine} we will extend the linearization procedure to the case of nonlinear connections defined on affine bundles. Examples and applications will be given in Section~\ref{examples.applications} where the emphasis is in new proofs of our results in previous papers on the geometry of second-order differential equations. In Section~\ref{cotangent} we will study the peculiarities of nonlinear connections on cotangent bundles mainly related with the canonical symplectic structure. The paper ends with some conclusions and an outlook of future research.

\section{Preliminaries}
\label{preliminaries}


All objects in this paper are assumed in the $C^\infty$ category. The tangent bundle to a manifold $M$ will be denoted $\map{{\tau}_M}{TM}{M}$, and the cotangent bundle by $\map{{\pi}_M}{T^*M}{M}$. The $\cinfty{M}$-module of vector fields on $M$ will be denoted $\vectorfields{M}$. The Lie derivative with respect to a vector field $X\in \vectorfields{M}$ will be denoted by $\call_X$. The set of sections of a fiber bundle $\map{\nu}{B}{M}$ will be denoted by $\sec{B}$.

Given a map $\map{f}{N}{M}$ and a fiber bundle $\map{\nu}{B}{M}$, the pullback bundle $\map{\pr_1}{f^*B}{N}$ has total space $f^*B=\set{(n,b)\in N\times B}{\nu(b)=f(n)}$ and projection $\pr_1(n,b)=n$. A section of $\map{\pr_1}{f^*B}{N}$ can be canonically identified with a section of $B$ along $f$, that is, a map $\map{\sigma}{N}{B}$ such that $\nu\circ\sigma=f$. For the details see~\cite{Poor, sections.along.maps}. In what follows in this paper we will assume such identification, i.e. whenever we talk about a section of $f^*B$ we will think of it as a map from $N$ to $B$ as above, or in other words, the module of sections of the pullback bundle will be identified with
\begin{equation}
\label{identification.sections.along}
\sec{f^*B}\equiv\set{\map{\sigma}{N}{B}}{\nu\circ\sigma=f}.
\end{equation}
A section $\sigma\in\sec{f^*B}$ is said to be basic if there exists a section $\alpha\in\sec{B}$ such that $\sigma=\alpha\circ f$. We will frequently identify both sections, i.e. $\alpha$ indicates both the section of $B$ and the section of $f^*B$, and the context will make clear the meaning. When $f=\nu$, the pullback bundle $\map{\pr_1}{\nu^*B}{B}$ carries a canonical section $\cansec$ given by the identity map in the manifold $B$, that is $\cansec(b)=b$.

We will be mainly concerned with the case of a vector bundle $\map{\pi}{E}{M}$. The pullback $\map{\pr_1}{f^*E}{N}$ is also a vector bundle. In most cases of interest to us the map $f$ is just the projection $\pi$. By a (regular) distribution along $\pi$ we mean a subbundle of $\pi^*E$. A distribution $\cald$ along $\pi$ is said to be basic if there exists a subbundle $F\subset E$ such that $\cald=\pi^*F$. Equivalently, a distribution is basic if it is generated by basic sections.

Local coordinates $(x^i)$ defined on some open set $\calu$ of the base $M$ and a local basis $\{e_A\}$ of sections of $E|_\calu$ determine local coordinates $(x^i,u^A)$ on $E$ as well as local coordinates $(x^i,u^A,w^A)$ on $\pi^*E$. The bundle projection $\map{\pr_1}{\pi^*E}{E}$ reads $\pr_1(x^i,u^A,w^A)=(x^i,u^A)$. The canonical section has the local expression $\cansec=u^Ae_A$.

The vertical subbundle $\map{{\tau}_E^\V\equiv{\tau}_E\big|_{\Ver{E}}}{\Ver{E}=\Ker(T\pi)}{E}$ can be canonically identified with the pullback bundle $\map{\pr_1}{\pi^*E}{E}$ by means of the vertical lifting map $\map{\xi^\V}{\pi^*E}{TE}$, defined by $\xi^\V(a,b)=\frac{d}{ds}(a+sb)|_{s=0}$ for every $(a,b)\in\pi^*E$. For a section $\sigma\in\sec{\pi^*E}$ the vertical lift of $\sigma$ is the vector field $\sigma^\V\in\vectorfields{E}$ given by ${\sigma}^\V=\xi^\V\circ\sigma$. The vertical lift of the canonical section is the Liouville vector field $\Delta=\cansec^\V$, the infinitesimal generator of dilations along the fibers. 

Let $\map{{\nu}}{F}{N}$ be another vector bundle. Let $\map{\phi}{E}{F}$ be a (generally nonlinear) map fibered over $\map{\varphi}{M}{N}$. The fibre derivative of $\phi$ is the vector bundle map $\map{\calf\phi}{\pi^*E}{\nu^*F}$ over $\phi$ defined by the equation $\xi^\V\circ\calf\phi=T\phi\circ\xi^\V$. In local coordinates as above, if the local expression of $\phi$ is $\phi(x^i,u^A)=\bigl(\varphi^a(x^i),\phi^\alpha(x^i,u^A)\bigr)$, then the local expression of $\calf\phi$ is  
\[
\calf\phi(w^B e_B)=w^B\pd{\phi^\alpha}{u^B}(x^i,u^A)\,e_\alpha.
\]

If $N=M$ and $\map{\phi}{E}{F}$ is a vector bundle map over the identity in $M$, for a section $\sigma$ of $E$ along a map $f$ the section of $F$ along $f$ defined by composition $n\mapsto \phi(\sigma(n))$ will be denoted by $\phi(\sigma)$.

\subsection*{Connections}

For a vector bundle $\map{\pi}{E}{M}$ there is a short exact sequence of vector bundles over $E$
\begin{equation}
\label{fundamental.sequence}
\seq 0-> \pi^*E -\xi^\V-> TE -p_E-> \pi^*TM ->0,
\end{equation}
where $\map{p_E}{TE}{\pi^*(TM)}$ is the projection $p_E(V)=(\tau_E(V),T\pi(V))$. This sequence is known as the fundamental sequence of $\map{\pi}{E}{M}$.

A connection on the vector bundle $E$ is a (right) splitting of the fundamental sequence, that is, a map $\map{\xi^\H}{\pi^*TM}{TE}$ such that $p_E\circ\xi^\H=\id_{\pi^*(TM)}$. Equivalently, a connection is given by the associated left splitting of that sequence, that is, the map $\map{\kappa}{TE}{\pi^*E}$ such that $\kappa\circ\xi^\V=\id_{\pi^*E}$ and $\kappa\circ\xi^\H=0$. The map $\xi^\H$ is said to be the horizontal lifting and the map $\kappa$ is usually known as the connector map (also Dombrowski connection map). See~\cite{Poor} for the general theory of connections on vector bundles.

A connection decomposes $TE$ as a direct sum $TE=\Hor\oplus\Ver E$, where $\Hor=\Im(\xi^\H)=\Ker(\kappa)$ is said to be the horizontal subbundle. The projectors over the horizontal and vertical subbundles are given by $P_\H=\xi^\H\circ p_E$ and $P_\V=\xi^\V\circ\kappa$. For a vector field $X$ along $\pi$, i.e. a section $X\in\sec{\pi^*TM}$, the horizontal lift of $X$ is the vector field $X^\H\in\vectorfields{E}$ given by $X^\H=\xi^\H\circ X$. The curvature of the connection is the $E$-valued 2-form $\map{R}{\pi^*(TM\wedge TM)}{\pi^*E}$ defined by
\begin{equation}
R(X,Y)={\kappa}([X^\H,Y^\H]),\qquad\qquad X,Y\in\sec{\pi^*TM}.
\end{equation}

In a local coordinate system as above, a local basis for $\Hor$ is given by the local vector fields $H_i=(\partial/\partial x^i)^\H$. This vector fields have the expression
\begin{equation}
H_i=\pd{}{x^i}-{\Gamma}^A_i(x,u)\pd{}{u^A},
\end{equation}
for some smooth functions ${\Gamma}^A_i\in \cinfty{E}$, known as the connection coefficients in the given coordinate system. To complete a local basis of vector fields on $E$ we will use the vector fields $V_A=(e_A)^\V=\partial/\partial u^A$.  The components of the curvature are
\begin{equation}
R_{ij}^A=\left\langle e^A\,,\, R\Bigl(\pd{}{x^i},\pd{}{x^j}\Bigr)\right\rangle=H_j(\Gamma^A_i)-H_i(\Gamma^A_j),
\end{equation}
where $\{e^A\}$ is the dual basis of $\{e_A\}$.

For ${\lambda}\in\Real$ let $\map{\mu_\lambda}{E}{E}$ be the multiplication map, $\mu_\lambda(a)=\lambda a$. A connection $\xi^\H$ is said to be linear if $\xi^\H(\lambda a,v)=T\mu_\lambda(\xi^\H(a,v))$ for every $\lambda\in\Real$ and every $(a,v)\in\pi^*TM$. In this case the connection coefficients are linear functions in the fibre coordinates $\Gamma^A_i(x,u)=\Gamma^A_{iB}(x)u^B$. If the connection is not smooth on the zero section then the above condition warranties only homogeneity of the connection. For a homogeneous connection the connection coefficients are homogeneous functions with degree 1 on the fibre coordinates.

A linear connection on $\map{\pi}{E}{M}$ can be equivalently defined by a covariant derivative operator, that is an $\Real$-bilinear map $\map{D}{\vectorfields{M}\times\sec{E}}{\sec{E}}$, $(X,\sigma)\mapsto D_X\sigma$,  such that $D_{fX}\sigma=fD_X\sigma$ and $D_X(f\sigma)=X(f)\sigma+fD_X\sigma$ for every $f\in\cinfty{M}$, $X\in\vectorfields{M}$ and $\sigma\in\sec{E}$. In terms of the connection map the covariant derivative of a section $\sigma\in \sec{E}$ is given by $D_X{\sigma}={\kappa}(T{\sigma}(X))$. The connections coefficients $\Gamma^A_i(x,u)=\Gamma^A_{iB}(x)u^B$ can be recovered from the covariant derivative of the elements of the basis $\{e_A\}$ by means of $D_{\partial/\partial x^i}e_B=\Gamma^A_{iB}(x)e_A$. For a fixed ${\sigma}\in \sec{E}$ and every $m\in M$, the value of $D_X{\sigma}$ at $m$ depends only on the value of $X$ at $m$. Thus, for $v\in T_mM$ we can write $D_v{\sigma}=(D_X{\sigma})(m)$ for any $X\in \vectorfields{M}$ such that $X(m)=v$.

A map $\map{f}{N}{M}$ defines a connection on $f^*E$ by pullback. When the connection is linear the associated covariant derivative ${}^f\!D$ is determined by the relation ${}^f\!D_v(\sigma\circ f)=D_{Tf(v)}\sigma$ for every $v\in TN$ and every $\sigma\in\sec{E}$. In general it is not necessary to make any notational distinction between $D$ and ${}^f\!D$. 

For a curve $\map{\gamma}{[t_0,t_1]}{M}$ a section $\sigma$ of $E$ along $\gamma$ is parallel if, as a curve, it is horizontal, $\dot{\sigma}(t)\in\Hor_{\sigma(t)}$ for all $t\in[t_0,t_1]$. Equivalently $\sigma$ is parallel if $D_{\dot{\gamma}(t)}\sigma=0$. For a vector $a\in E_{\gamma(t_0)}$, if $\sigma(t)$ is the solution of the initial value problem $D_{\dot{\gamma}(t)}\sigma=0$, $\sigma(t_0)=a$, then the parallel transport of $a$ along $\gamma$ is the vector $\sigma(t_1)\in E_{\gamma(t_1)}$. If the curve $\gamma$ is an integral curve of a vector field $X\in\vectorfields{M}$ then the parallel section $\sigma$ with initial value $a$ is $\sigma(t)=\phi_t(a)$, where $\phi_t$ is the flow of the horizontal lift $X^\H\in\vectorfields{E}$ of $X$.

A linear connection on the tangent bundle $\map{\tau_M}{TM}{M}$ to a manifold $M$ is said to be a linear connection on the manifold $M$. 

\begin{remark}
In many occasions one has to consider connections which are smooth except on a submanifold $C\subset E$ which contains the zero section. A frequent case is that of homogeneous connections, i.e. connections that satisfy the property $\xi^\H(\lambda p,v)=T\mu_\lambda(\xi^\H(p,v))$ for all $\lambda\in\Real-\{0\}$ and all $(p,v)\in \pi^*E$, with $p\not\in C$. In such cases what we have is a horizontal distribution on $TE|_{E-C}\to E-C$. This possibility  will be assumed but it will not be indicated explicitly in the notation. 
\end{remark}



\section{Linearization of a non-linear connection on a vector bundle}
\label{linearization}
A nonlinear connection on a vector bundle $\map{\pi}{E}{M}$ defines a linear connection on the pullback bundle $\map{\pr_1}{\pi^*E}{E}$ which will be called the linearization of the nonlinear connection. We will define such linear connection by means of the associated covariant derivative.

\begin{theorem}
\label{covariant.derivative}
The map $\map{D}{\vectorfields{E}\times\sec{{\pi}^*E}}{\sec{{\pi}^*E}}$, given by 
\begin{equation}
\label{linearization.general}
D_U{\sigma}={\kappa}\Bigl([P_\H(U),{\sigma}^\V]+T{\sigma}(P_\V(U))\Bigr)
\end{equation}
is a covariant derivative on the vector bundle $\map{\pr_1}{{\pi}^*E}{E}$.
\end{theorem}
\begin{proof}
It is clear that $D_U\sigma$ is a well-defined section of $\pi^*E$ and that it is ${\Real}$-linear in both arguments, so  we just need to prove that it is $\cinfty{E}$-linear on the first argument and a derivation in the second. For any function $f\in\cinfty{E}$ we have
\begin{align*}
D_{fU}\sigma
&={\kappa}\Bigl([fP_\H(U),{\sigma}^\V]+fT{\sigma}(P_\V(U))\Bigr)\\
&={\kappa}\Bigl(-\sigma^\V(f)P_H(U)+f[P_\H(U),{\sigma}^\V]+fT{\sigma}(P_\V(U))\Bigr)\\
&=f{\kappa}\Bigl([P_\H(U),{\sigma}^\V]+T{\sigma}(P_\V(U))\Bigr)\\
&=fD_U\sigma,
\end{align*}
where we have used ${\kappa}\circ P_H=0$. On the other hand,
\begin{align*}
D_{U}(f\sigma)
&={\kappa}\Bigl([P_\H(U),f{\sigma}^\V]+T(f{\sigma})(P_\V(U))\Bigr)\\
&={\kappa}\Bigl(P_\H(U)(f)\sigma^\V+ f[P_\H(U),{\sigma}^\V]+fT{\sigma}(P_\V(U)) +P_\V(U)(f)\sigma^\V\Bigr)\\
&=\bigl(P_\H(U)(f)+P_\V(U)(f)\bigr){\kappa}(\sigma^\V)+f{\kappa}\Bigl([P_\H(U),{\sigma}^\V]+T{\sigma}(P_\V(U))\Bigr)\\
&=U(f)\sigma+fD_U\sigma,
\end{align*}
where we have used $T(f\sigma)(V)=V(f)\sigma^\V+fT\sigma(V)$, and ${\kappa}(\sigma^\V)=\sigma$.
\end{proof}

Since any vector field $U\in\vectorfields{E}$ can be written in the form  $U=X^\H+\eta^\V$ for $X\in\sec{\pi^*TM}$ and $\eta\in\sec{\pi^*E}$ the linear connection $D$ is determined by 
\begin{equation}
\label{basic.covariant.derivative.rules}
D_{X^\H}\sigma={\kappa}([X^\H,\sigma^\V])
\qquad\text{and}\qquad
D_{\eta^\V}\sigma={\kappa}(T\sigma(\eta^\V)).
\end{equation}

The horizontal distribution associated to the covariant derivative $D$ will be denoted $\overline{\Hor}$, and accordingly the horizontal lifting map by $\xi^{\bar{\H}}$.

\begin{remark}
When the connection $\Hor$ is linear with covariant derivative $\nabla$ then the covariant derivative obtained by linearization of $\Hor$ is just the pullback of $\nabla$ by $\pi$, as it is indicated by~\eqref{basic.covariant.derivative.rules}. 
\end{remark}


To finish this subsection, we notice that, as any linear connection, the connection $D$ on $\map{\pr_1}{\pi^*E}{E}$ has an associated dual linear connection on the dual vector bundle $\map{\pr_1}{\pi^*E^*}{E}$.  

\subsection*{Connection coefficients}

Consider local coordinates $(x^i)$ on $M$ and a local basis $\{e^A\}$ of sections of $E$. Let $(x^i,u^A)$ be the induced local coordinate system on $E$ and $(x^i,u^A,w^A)$ the induced local coordinate system on $\pi^*E$. The covariant derivative of the elements of the basis $\{e_A\}$ is given by 
\begin{equation}
\label{linear.connection.coefficients}
\begin{aligned}
D_{H_i}e_A&=\pd{{\Gamma}^B_i}{u^A}(x,u)e_B,\\[5pt]
D_{V_B}e_A&=0.
\end{aligned}
\end{equation}
As we see, the connection coefficients of the linear connection are the linearization of the nonlinear connection coefficients along the directions of the fibre. 

\smallskip

\begin{remark}
In what follows in this paper we will use the notation 
\begin{equation}
{\Gamma}^A_{iB}=\partial {\Gamma}^A_i/\partial u^B.
\end{equation}
\end{remark}
\smallskip

For a section ${\sigma}={\sigma}^A(x,u)e_A$ and a vector field $U=U^iH_i+U^AV_A$ we have
\begin{equation}
D_U{\sigma}=\bigl\{U^i[H_i({\sigma}^A)+{\Gamma}^A_{iB}{\sigma}^B]+ U^BV_B({\sigma}^A)\bigr\}e_A.
\end{equation}
The horizontal distribution $\overline{\Hor}\subset T(\pi^*E)$ associated to the linear connection is locally spanned by the horizontal lift $\{\bar{H}_i,\bar{H}_A\}$ of the coordinate vectorfields $\{\partial/\partial x^i, \partial/\partial u^A\}$,
\begin{equation}
\begin{aligned}
\bar{H}_i&=\left(\pd{}{x^i}\right)^{\bar{\H}}=\pd{}{x^i}-{\Gamma}^A_{iB}(x,u)w^B\pd{}{w^A},\\
\bar{H}_A&=\left(\pd{}{u^A}\right)^{\bar{\H}}=\pd{}{u^A}.
\end{aligned}
\end{equation}
The annihilator of $\overline{\Hor}$ is generated by the 1-forms $dw^A+{\Gamma}^A_{iB}(x,u)dx^i$.

\subsection*{Homogeneity}
The covariant derivative of the canonical section carries important information about the homogeneity of the connection. For a vertical vector field $\eta^\V$ we have 
\begin{equation}
D_{\eta^\V}\cansec=\eta.
\end{equation}
Indeed, $
D_{\eta^\V}\cansec=\kappa(T\cansec(\eta^\V))=\kappa(\eta^\V)=\eta$,
where we have taken into account that, as a section along $\pi$, the canonical section is the identity map in $E$.

For a horizontal vector field $X^\H$ we have
\[
D_{X^\H}\cansec=\kappa([X^\H,\cansec^\V])=\kappa([X^H,\Delta])=-\kappa(\call_\Delta X^\H).
\]
Therefore $D_{X^\H}\cansec$ measures the homogeneity of the nonlinear connection. When the connection is smooth everywhere it is well known that homogeneity is equivalent to linearity. Therefore we have the following result.

\begin{proposition}
\label{homogeneous}
The connection $\Hor$ is homogeneous if and only if $D_{X^\H}\cansec=0$. If the connection is smooth on the zero section, then it is linear if and only if $D_{X^\H}\cansec=0$. 
\end{proposition}
\begin{proof}
From $D_{X^\H}\cansec=-\kappa(\call_\Delta X^\H)$ it follows that $D_{X^\H}\cansec=0$ for all $X$ if and only if the horizontal distribution is invariant under the flow of $\Delta$, i.e. under the diffeomorphisms $D_s(a)=\mathrm{e}^{s}a$. In other words, if and only if $\Hor$ is invariant under multiplication by positive real numbers.
\end{proof}

\begin{definition}
The $E$-valued 1-form $\bs{t}$ along $\pi$ defined by $\bs{t}(X)=-D_{X^\H}\cansec$ is called the tension tensor of the connection.
\end{definition}

In local coordinates the tension reads $\bs{t}=(\Gamma^A_i-\Gamma^A_{iB}u^B)e_A\otimes dx^i$.

\subsection*{Parallel transport}
The parallel transport defined by the linearized connection can be easily related to the geometry of the nonlinear connection, as~\eqref{basic.covariant.derivative.rules} reveals. The following result generalizes those in~\cite{Crampin.Berwald,MikeFest}

\begin{theorem}
\label{parallel.transport}
The linearized connection is characterized by the following parallel transport rules:
\begin{itemize}
\item Parallel transport along a vertical curve is the standard complete parallelism on the vertical fibre (as a vector space).
\item For a horizontal curve embedded in the flow ${\phi}_t$ of the horizontal lift $X^\H$ of a basic vector field $X\in \vectorfields{M}$ the parallel transport map is $\psi_t={\kappa}{\circ}T{\phi}_t{\circ}\vlift$. 
\end{itemize}
\end{theorem}
\begin{proof}
Let $\map{\gamma}{[t_0,t_1]}{E}$ be a curve contained in the fibre $E_m=\pi^{-1}(m)$ for some $m\in M$, i.e. $\pi(\gamma(t))=m$ for all $t\in[t_0,t_1]$. Let $p_0=\gamma(t_0)$ and $p_1=\gamma(t_1)$ be the endpoints of $\gamma$. A section $\sigma\in\sec{\pi^*E}$ is parallel along $\gamma$ if $D_{\dot{\gamma}(t)}\sigma=0$ for all $t\in[t_0,t_1]$. Since $\dot{\gamma}(t)$ is vertical, from the definition of the linear connection $D$ we have that $D_{\dot{\gamma}(t)}\sigma=\kappa(T\sigma(\dot{\gamma}))$. It follows that $\sigma$ is parallel if and only if $\kappa\bigl(\frac{d}{dt}(\sigma\circ\gamma)(t)\bigr)=0$. The curve $\sigma\circ\gamma$ takes values on the fibre $E_m$ from where it follows that the tangent vector to this curve is vertical and hence $\kappa\bigl(\frac{d}{dt}(\sigma\circ\gamma)(t)\bigr)=0$ if and only if $\frac{d}{dt}(\sigma\circ\gamma)(t)=0_{\gamma(t)}$ for all  $t\in[t_0,t_1]$. Therefore $\sigma\circ\gamma$ is constant and the parallel transport of $(\gamma(t_0),\sigma(p_0))$ along $\gamma$ is $(\gamma(t_1),\sigma(p_1))=(\gamma(t_1),\sigma(p_0))$. Thus, the parallel transport rule along curves in $E_m$ is independent of the curve $\gamma$ that joins the endpoint $p_0$ and $p_1$, and hence a complete parallelism on $E_m$. 

Let $X$ be a vector field on $M$ and $X^\H\in\vectorfields{E}$ its horizontal lift with respect to the given nonlinear connection. Denote by $\varphi_t$ the flow of $X$ and by $\phi_t$ the flow of $X^\H$, so that $\pi\circ\phi_t=\varphi_t\circ\pi$. Consider the family of maps $\psi_s=\kappa\circ T\phi_s\circ\xi^\V$. We will show that $\psi_s$ is a local flow on $\pi^*E$ which projects to $\phi_s$. Indeed, for $s=0$ we have $\psi_0=\kappa\circ\xi^\V=\id_{\pi^*E}$ and for $s,t\in\Real$ (as long as all terms are defined)
\begin{align*}
\psi_s\circ\psi_t
&=\kappa\circ T\phi_s\circ\xi^\V\circ\kappa\circ T\phi_t\circ\xi^\V
=\kappa\circ T\phi_s\circ P_\V \circ T\phi_t\circ\xi^\V={}\\
&=\kappa\circ T\phi_s \circ T\phi_t\circ\xi^\V
=\kappa\circ T\phi_{s+t}\circ\xi^\V
=\psi_{s+t},
\end{align*}
where we have used that $P_\V=\xi^\V\circ\kappa$ and that $T\phi_s\circ\xi^\V$ takes vertical values due to the projectability of the flow $\phi_s$. Moreover, 
\[
\pr_1\circ \psi_s
=\pr_1\circ\kappa\circ T\phi_s\circ\xi^\V
=\tau_E\circ T\phi_s\circ\xi^\V
=\phi_s\circ\tau_E\circ\xi^\V
=\phi_s\circ\pr_1,
\]
so that $\psi_s$ projects to $\phi_s$. 

To end the proof we just need to show that $D_{X^\H}\sigma=\frac{d}{ds}(\psi_{-s}\circ\sigma\circ\phi_s)\at{s=0}$, so that $\psi_s$ is the flow of the horizontal lift for the linearized connection of $X^\H$, and hence the parallel transport map along the flow of $X^\H$. First, notice that
\[
\psi_{-s}\circ\sigma\circ\phi_s
=\kappa\circ T\phi_{-s}\circ\xi^\V\circ\sigma\circ\phi_s
=\kappa\circ T\phi_{-s}\circ\sigma^\V\circ\phi_s.
\]
If $p\in E$, the curve $s\mapsto T\phi_{-s}\circ\sigma^\V\circ\phi_s(p)$ takes values on the vector space $E_\pi(p)$ so that taking the derivative 
\begin{align*}
\frac{d}{ds}(\psi_{-s}\circ\sigma\circ\phi_s)(p)\at{s=0}
&=\frac{d}{ds}(\kappa(T\phi_{-s}(\sigma^\V(\phi_s(p)))))\at{s=0}\\
&=\kappa\left(\frac{d}{ds}(T\phi_{-s}(\sigma^\V(\phi_s(p))))\at{s=0}\right)\\
&=\kappa(\call_{X^\H}\sigma^\V(p))\\
&=\kappa([X^\H,\sigma^\V])(p)\\
&=(D_{X^\H}\sigma)(p).
\end{align*}
As this relation holds true for every $p\in E$, the second statement is proved.

The given rules of parallel transport determine the connection, as equations~\eqref{basic.covariant.derivative.rules} determine the covariant derivative~$D$.
\end{proof}

\begin{remark}
Theorem~\ref{parallel.transport} expresses in a different manner in what sense the connection $D$ is the linearization of the nonlinear connection. When the connection is linear, the flow $\phi_s$ of $X^\H$ is a linear map in $E$, which is but the parallel transport map along the integral curves of $X$. When the connection is nonlinear, we can linearize the flow by means of
\[
\frac{d}{dt}\phi_s(p+ta)\at{t=0}=T\phi_s(\xi^\V(p,a))
\]
which is a vertical vector to whom we apply the inverse of $\xi^\V$ (in other words $\kappa$) to get a vector back in the fibre $E_{\pi(p)}$. In other words the parallel transport along horizontal curves is the fibre derivative $\psi_s=\calf\phi_s$ of the the flow of $X^\H$.
\end{remark}

A frequent problem one encounters is to decide if a given section along $\pi$ or a given distribution along $\pi$ is basic. We recall that a section $\sigma\in\sec{\pi^*E}$ is said to be basic if there exists a section $\beta\in\sec{E}$ such that $\sigma=\beta\circ\pi$. A distribution $\cald$ along $\pi$ is said to be basic if there exists a subbundle $F\subset E$ such that $\cald=\pi^*F$.  

\begin{proposition}
\label{characterization.basic}
A section $\sigma\in\sec{\pi^*E}$ is basic if and only if $D_{\eta^\V}\sigma=0$ for every $\eta\in\sec{\pi^*E}$. A distribution $\cald\subset\pi^*E$ is basic if and only if $D_{\eta^\V}\cald\subset\cald$ for every $\eta\in\sec{\pi^*E}$.
\end{proposition}
\begin{proof}
If $\sigma$ is basic it is of the form $\sigma=\beta\circ\pi$ for some section $\beta\in\sec{E}$. Then $D_{\eta^\V}\sigma=\kappa(T\beta\circ T\pi(\eta^\V))=0$ for all $\eta\in\sec{\pi^*E}$. Conversely, if $D_{\eta^\V}\sigma=0$ for all $\eta\in\sec{\pi^*E}$ we have that $\sigma$ is parallel along any vertical curve. From Proposition~\ref{parallel.transport} it follows that $\sigma(a)=\sigma(b)$ for all $a,b\in E_m$ and all $m\in M$. Defining $\beta(m)=\sigma(0_m)$ we have that $\sigma=\beta\circ\pi$.

If $\cald$ is basic there exists a finite family $\{\beta_I\}$ of sections of $E$ such that any section $\sigma$ in $\cald$ is of the form $\sigma=\sum_If_I(\beta_I\circ\pi)$ for some functions $f_I\in\cinfty{E}$. Then $D_{\eta^\V}\sigma=\sum_I\eta^\V(f_I)(\beta_I\circ\pi)$ is a section of $\cald$. Conversely, if $D_{\eta^\V}\cald\subset\cald$ for all $\eta\in\sec{\pi^*E}$ we have that $\cald$ is invariant under parallel transport along any vertical curve. For $a\in E$ let $D(a)\subset E_{\pi(a)}$ be the vector subspace $D(a)=\pr_2(\cald(a))$, so that $\cald(a)=(a,D(a))$. From Proposition~\ref{parallel.transport} it follows that $D(a)=D(b)$ for all $a,b\in E$ such that $\pi(a)=\pi(b)$. Defining $F(m)=D(0_m)$ for every $m\in M$ we have that $\cald=\pi^*F$.
\end{proof}

Due to the linearity in the first argument of the covariant derivative, the above properties can be equivalently stated as follows. A section ${\sigma}\in \sec{{\pi}^*E}$ is basic if and only if $D_u{\sigma}=0$ for all $u\in \Ver{{\pi}}$. A distribution $\cald{\subset}{\pi}^*E$ is basic if and only if $D_u{\sigma}\in \cald({\tau}_E(u))$ for all ${\sigma}\in\sec{\cald}$ and all $u\in \Ver{{\pi}}$.

\subsection*{Natural prolongation}
\label{natural.operations.vector}
The procedure described in this paper, which assigns a connection on the pullback bundle ${\pi}^*E$ to a non linear one on the bundle $E$, can be properly understood as a prolongation. The pullback bundle $\map{\pr_1}{{\pi}^*E}{E}$ can be identified (via the vertical lifting) with the vertical bundle $\map{{\tau}_E^\V\equiv{\tau}_E|_{\Ver{E}}}{\Ver{E}}{E}$. Under this identification the linearized connection can be considered as a connection on the vertical bundle. 

A systematic study of of the functorial prolongation of a connection on a (non necessarily linear) bundle $B\to M$ to a connection on the vertical bundle has been carried out in~\cite{DouMi2, DouMi1, DouMi3}. In~\cite{DouMi1} it is proved that a bundle functor $G$ on the category of fibred manifolds admits a functorial operator lifting connections on $B\to M$ to connections on $GB\to B$ if and only if the functor $G$ is a trivial bundle functor $B\mapsto GB=B\times W$, for some manifold $W$. As a consequence of this result, since the vertical functor is not a trivial functor, we deduce that there is no natural operator transforming connections on $B\to M$ to connections on $\Ver{B}\to B$. However, as it is already remarked in~\cite{DouMi1}, the obstruction for the existence of such a natural operator may disappear if we consider some additional structure. In~\cite{NatOp}, by restricting to the subcategory of affine bundles, it is shown that there exists only a 1-parameter family of first order natural operators (natural on local isomorphisms of affine bundles) transforming connections on $\map{{\pi}}{E}{M}$ into connections on $\map{{\tau}_E^\V}{\Ver{E}}{E}$. 

Our linearization procedure transforms a connection on a vector bundle (a particular case of an affine bundle) into a linear connection on the pullback bundle, or equivalently to the vertical bundle. From the expression of the connection coefficients of the linearized connection it is clear that our construction is a first order differential operator, and from the definition (see equation~\eqref{linearization.general}) it is easy to see that it is a natural operation, as it is constructed geometrically only in terms of brackets, the vertical lifting and the connection itself. Therefore, it should coincide with one of the members of the above mentioned family.

With the notation in this paper, the 1-parameter family of connections $\overline{\Hor}_{\lambda}$, ${\lambda}\in {\Real}$, described in~\cite{NatOp} (section~\textbf{46.10}) is given as the annihilator of the 1-forms
\[
dw^A+{\Gamma}^A_{iB}(x,u)w^Bdx^i+{\lambda}\bigl(du^A+{\Gamma}^A_i(x,u)dx^i\bigr).
\]
However, no explicit intrinsic construction is provided in~\cite{NatOp}. Inspired by the result of Theorem~\ref{parallel.transport}, a long but straightforward coordinate calculation shows that the horizontal lifting $\map{{\xi}^{\bar\H_{\lambda}}}{({\tau}_E^\V)^*TE}{T\Ver{E}}$ associated to the connection $\overline{\Hor}_{\lambda}$ on $\Ver{E}\to E$ is given by 
\[
{\xi}^{\bar\H_{\lambda}}(v,w)={\chi}_E\Bigl(T\hlift\bigl(v,0_{T{\pi}(w)}\bigr)\Bigr)
+
T\vlift\Bigl(w-\hlift\bigl(a,T{\pi}(w)\bigr), \vlift\bigl(b,{\lambda}{\kappa}(w)\bigr) \Bigr),
\] 
where $\map{{\chi}_E}{TTE}{TTE}$ is the canonical involution and $(a,b)=(\vlift)^{-1}(v)$. Our connection corresponds to ${\lambda}=0$ and it is the only one in the family which is linear (all the other are affine connections).

We can resume the contents of this subsection by saying that the linearized connection defined in this paper is the only linear connection on the pullback bundle (or equivalently in the vertical bundle) that can be obtained as a first order natural prolongation of a general connection on a vector bundle, and that for general bundles a similar construction is not possible.

\subsection*{Curvature} The curvature tensor of the linear connection $D$ is defined by 
\begin{equation}
\curv(U,V)\sigma=[D_U,D_V]\sigma-D_{[U,V]}\sigma,
\end{equation}
for $U,V\in\vectorfields{E}$ and $\sigma\in\sec{\pi^*E}$. It is clear that this tensor must be related to the curvature of the initial nonlinear connection. 

We first notice that the $\scriptstyle\mathsf{VV}$-component vanishes
\begin{equation}
\curv(\zeta^\V,\eta^\V)\sigma=0.
\end{equation} 
This follows immediately from the parallel transport rule along vertical curves. A more direct proof is as follows. $\curv$ is a tensor, so that we just need to show that it vanishes when applied to the elements of some generating family. Taking a basic section $\sigma\in\sec{E}$ then $D_{\eta^\V}\sigma=0$, $D_{\zeta^\V}\sigma=0$ and $D_{[\zeta^\V,\eta^\V]}\sigma=0$, from where $\curv(\zeta^\V,\eta^\V)\sigma=0$ follows. 

For the $\scriptstyle\mathsf{VH}$-component $\curv(\xi^\V,Y^\H)\sigma$ we have the coordinate expression
\begin{equation}
\curv(V_B,H_i)e_A=\pd{\Gamma^C_{iA}}{u^B}\,e_C=\pd{^2\Gamma^C_{i}}{u^A\partial u^B}\,e_C,
\end{equation}
which is clearly symmetric in the indices $A, B$. It follows that 
\begin{equation}
\label{VH.is.symmetric}
\curv(\eta^\V,X^\H)\sigma=\curv(\sigma^\V,X^\H)\eta.
\end{equation}

Finally, for the $\scriptstyle\mathsf{HH}$-component we have the coordinate expression  
\begin{equation}
\curv(H_i,H_j)e_A=\left[
H_i(\Gamma^B_{jA}) - H_j(\Gamma^B_{iA}) + \Gamma^B_{iC}\Gamma^C_{jA} - \Gamma^B_{jC}\Gamma^C_{iA}
\right]\,e_B.
\end{equation}
The term $H_i(\Gamma^B_{jA})$ can be written as 
\begin{align*}
H_i(\Gamma^B_{jA})
&=H_i(V_A(\Gamma^B_{j}))\\
&=[H_i,V_A](\Gamma^B_{j})+V_A(H_i(\Gamma^B_{j}))\\
&=\Gamma_{iA}^CV_C(\Gamma^B_{j})+V_A(H_i(\Gamma^B_{j}))\\
&=\Gamma_{iA}^C\Gamma^B_{jC}+V_A(H_i(\Gamma^B_{j})),
\end{align*}
from where 
\begin{equation}
\label{curv.is.Dv.R.local}
\curv(H_i,H_j)e_A=[V_A(H_i(\Gamma^B_{j}))-V_A(H_j(\Gamma^B_{i}))]\,e_B=-V_A(R^B_{ij})\,e_B.
\end{equation}
Therefore the $\scriptstyle\mathsf{HH}$-component provides a linearization along the fibre of the curvature of the nonlinear connection. Later on we will rewrite this relation in intrinsic terms.

The following result gives an interpretation of the $\scriptstyle\mathsf{VH}$-component of the curvature of the connection $D$.
\begin{proposition}
\label{interpretation.VH}
The $\scriptstyle\mathsf{VH}$-component of the curvature, $(\eta,X)\mapsto\curv(\eta^\V,X^\H)$, vanishes if and only if there exists a linear connection $\nabla$ on $\map{\pi}{E}{M}$ such that $D$ is the pullback of $\nabla$ by $\pi$. In such case the curvature tensor of $\nabla$ is $\curv^\nabla(X,Y)\sigma=\curv(X^\H,Y^\H)\sigma$ for all $X,Y\in\vectorfields{M}$ and $\sigma\in\sec{E}$.
\end{proposition}
\begin{proof}
Let $D$ be the pullback by $\pi$ of a linear connection $\nabla$ on $E$. For every $X\in\vectorfields{M}$ and every $\sigma\in\sec{E}$ we have that $D_{X^\H}\sigma$ is basic, concretely $D_{X^\H}\sigma=(\nabla_X\sigma)\circ\pi$. It follows that $D_{\eta^\V}\sigma$, $D_{\eta^\V}D_{X^\H}\sigma$ vanish for every $\eta\in\sec{E}$. But also $D_{[\eta^\V,X^\H]}\sigma$ vanishes because $[\eta^\V,X^\H]$ is vertical, from where $\curv(\eta^\V,X^\H)\sigma=D_{\eta^\V}D_{X^\H}\sigma-D_{X^\H}D_{\eta^\V}\sigma-D_{[\eta^\V,X^\H]}\sigma=0$.

Conversely, let $X\in\vectorfields{M}$ and $\sigma\in\sec{E}$. We will prove that if $\curv(\eta^\V,X^\H)\sigma=0$ for all $\eta\in\sec{E}$ then $D_{X^\H}\sigma$ is basic. Indeed, since $D_{\eta^\V}\sigma=0$,
\[
D_{\eta^\V}(D_{X^\H}\sigma)
=D_{\eta^\V}D_{X^\H}\sigma-D_{X^\H}D_{\eta^\V}\sigma
=D_{[\eta^\V,X^\H]}\sigma
=0
\]
where we have used that $[\eta^\V,X^\H]$ is a vertical vector field.

It follows that there exists a section $\nabla_X\sigma\in\sec{E}$ such that $D_{X^\H}\sigma=(\nabla_X\sigma)\circ\pi$. The operator $\nabla$ is a linear connection due to the properties of $D$. Also $D_{\eta^\V}\sigma=0$, and hence $D$ is the pullback by $\pi$ of $\nabla$.

The formula for the curvature follows from the structure equations for the curvature tensor (see~\cite{Poor}, paragraph 2.65).
\end{proof}

\begin{remark}
In fact the vanishing of the $\scriptstyle\mathsf{VH}$-component of the curvature characterizes those connections which are affine (in the sense explained bellow in Section~\ref{linearization.affine}). The difference between the connection $\Hor$ and the connection $\nabla$ in Proposition~\ref{interpretation.VH} is given by the tension tensor, which in this case is a basic tensor field. If the connection is smooth on the zero section then it is easy to see that if the tension is a basic tensor field then the connection is affine and the $\scriptstyle\mathsf{VH}$-component of the curvature vanishes.
\end{remark}

\paragraph{Bianchi identities} 
Any linear connection satisfies a set of identities, known as Bianchi identities, which are deduced from the Jacobi identity for the covariant derivatives, $\sum[D_U,[D_V,D_W]]=0$ where the sum is over cyclic permutations of $U,V,W$. They can be expressed as
\[
\sum\Bigl(
[D_U\curv(V,W)]\sigma+\curv(U,[V,W])\sigma
\Bigr)
=0.
\]
Written in this form they do not look very appealing but they carry important information about derivatives of the different components of the curvature. 

To obtain more explicit expressions it is convenient to define the following tensors along $\pi$ which determine the curvature tensor $\curv$:
\begin{itemize}
\item $\map{\Rie}{\pi^*(TM\otimes TM\otimes E)}{\pi^*E}$ by 
\[
\Rie(X,Y)\sigma=\curv(X^\H,Y^\H)\sigma,
\]
\item $\map{\theta}{\pi^*(E\otimes E\otimes TM)}{\pi^*E}$ by 
\[
\theta(\sigma,\zeta)Y=\curv(\zeta^\V,Y^\H)\sigma.
\]
\end{itemize}
Notice that $\theta$ is symmetric while $\Rie$ is skewsymmetric.

In order to find simpler expressions of the Bianchi identities we will use an auxiliary linear connection on $\pi^*TM\to E$. We will assume that the covariant derivative $\bs{D}$ of such connection satisfies the following natural condition: for every vertical vector field $V\in\vectorfields{E}$
\begin{equation}
\label{natural.connection}
\bs{D}_{V}Z=T\pi([V,Z^\H]),\qquad\forall Z\in\sec{\pi^*TM}.
\end{equation}
Connections satisfying this property will be said compatible. For instance, if $\nabla$ is a linear connection on the base manifold $M$ then the connection $\bs{D}_{U}Z=\nabla_{P_\H(U)}Z+T\pi([P_\V(U),Z^\H])$ is a connection on $\pi^*TM\to E$ satisfying such condition. 
The torsion tensor $T$ of the connection $\bs{D}$ is the tensor field defined by 
\[
T(X,Y)=\bs{D}_{X^\H}Y-\bs{D}_{Y^\H}X-[X,Y],\qquad\forall X,Y\in\vectorfields{M}.
\]  

Once fixed a compatible connection we can conveniently write the different brackets of vector fields in $E$ in terms of covariant derivatives:
\begin{equation}
\label{brackets}
\begin{aligned}
&[X^\H,Y^\H]=(\bs{D}_{X^\H} Y)^\H-(\bs{D}_{Y^\H} X)^\H-T(X,Y)^\H+R(X,Y)^\V\\
&[X^\H,\,\eta^\V\,]=(D_{X^\H} \eta)^\V-(\bs{D}_{\eta^\V} X)^\H\\
&[\,\zeta^\V\,,\,\eta^\V\,]=(D_{\zeta^\V} \eta)^\V-(D_{\eta^\V} \zeta)^\V.
\end{aligned}
\end{equation}

We will use the following notation $D^\V_\eta\sigma=D_{\eta^\V}\sigma$, $D^\V_\eta Z=\bs{D}_{\eta^\V}Z$, and similarly $D^\H_X\sigma=D_{X^\H}\sigma$, $D^\H_X Z=\bs{D}_{X^\H}Z$, for $\eta,\sigma\in\sec{\pi^*E}$ and $X,Z\in\sec{\pi^*TM}$. In this way the covariant operators $D^\H$ and $D^\V$ can act over tensor fields along $\pi$ with mixed arguments in $E$ and $TM$. For instance, equation~\eqref{curv.is.Dv.R.local} can be written simply as 
\begin{equation}
\label{curv.is.Dv.R} 
\Rie(X,Y)\sigma=-(D^\V_\sigma R)(X,Y),
\end{equation}
which exactly means $\Rie(X,Y)\sigma=-D_{\sigma^\V}\bigl(R(X,Y)\bigr)+R(\bs{D}_{\sigma^\V}X,Y)+R(X,\bs{D}_{\sigma^\V}Y)$.

Using this machinery, a long but straightforward calculation shows that the Bianchi identities can be written in the form 
\begin{equation}
\label{Bianchi.identities}
\begin{gathered}
\sum\bigl\{(D^\H_X\Rie)(Y,Z)\sigma-\Rie(X,T(Y,Z))-\theta(\sigma,R(Y,Z))X\bigr\}=0\\[3pt]
(D^\V_\eta\Rie)(X,Y)\sigma=(D^\H_X\theta)(\sigma,\eta)Y-(D^\H_Y\theta)(\sigma,\eta)X+\theta(\sigma,\eta)T(X,Y)\\[5pt]
(D^\V_\eta\theta)(\sigma,\zeta)X=(D^\V_\zeta\theta)(\sigma,\eta)X.
\end{gathered}
\end{equation}

It is important to notice that, while the Bianchi identities depend only on the connection $D$, this later expressions depend also on the choice of the connection $\bs{D}$, as the operator $D^\H$ does. On the other hand, notice that the operator $D^\V$ does not depend on the choice of the connection $\bs{D}$, as long as condicion~\eqref{natural.connection} is satisfied.

\medskip

In terms of the operators $D^\V$ and $D^\H$ the tension and the curvature are related by the following identities.
\begin{proposition}
\label{derivatives.of.the.tension}
The tension satisfies the following properties
\begin{align}
&(D^\V_\zeta\bs{t})(X)=\theta(\cansec,\zeta)X\\
&(D^\H_X\bs{t})(Y)-(D^\H_Y\bs{t})(X)+\bs{t}(T(X,Y))+R(X,Y)+\Rie(X,Y)\cansec=0,
\end{align}
for $X,Y\in\sec{\pi^*TM}$ and $\zeta\in\sec{\pi^*E}$.
\end{proposition}
\begin{proof}
We have 
\begin{align*}
\theta(\cansec,\zeta)X
&=\curv(\zeta^\V,X^\H)\cansec\\
&=D_{\zeta^\V}D_{X^\H}\cansec-D_{X^\H}D_{\zeta^\V}\cansec-D_{[\zeta^\V,X^\H]}\cansec\\
&=-D_{\zeta^\V}(\bs{t}(X))-D_{X^\H}\zeta+D_{(D_{X^\H}\zeta)^\V-(\bs{D}_{\zeta^\V}X)^\H}\cansec\\
&=-D_{\zeta^\V}(\bs{t}(X))-\bs{t}(\bs{D}_{\zeta^\V}X)\\
&=D^\V_\zeta(\bs{t}(X))-\bs{t}(D^\V_\zeta X)
\end{align*}
which proves the first relation.

The curvature term $\Rie(X,Y)\cansec$ is equal to
\begin{align*}
\Rie(X,Y)\cansec
&=\curv(X^\H,Y^\H)\cansec\\
&=D_{X^\H}D_{Y^\H}\cansec-D_{Y^\H}D_{X^\H}\cansec-D_{[X^\H,Y^\H]}\cansec\\
&=-D_{X^\H}(\bs{t}(Y))+D_{Y^\H}(\bs{t}(X))-D_{[X^\H,Y^\H]}\cansec.
\end{align*}
Using $[X^\H,Y^\H]=\{\bs{D}_{X^\H}Y-\bs{D}_{Y^\H}X-T(X,Y)\}^\H+R(X,Y)^V$ we get that the term $D_{[X^\H,Y^\H]}\cansec$ is equal to
\[
D_{[X^\H,Y^\H]}\cansec=-\bs{t}(\bs{D}_{X^\H}Y)+\bs{t}(\bs{D}_{Y^\H}X)+\bs{t}(T(X,Y))+R(X,Y).
\]
Thus
\[
\Rie(X,Y)\cansec=-D^\H_X(\bs{t}(Y))+D^\H_Y(\bs{t}(X))+\bs{t}(D^\H_XY)-\bs{t}(D^\H_YX)-\bs{t}(T(X,Y))-R(X,Y),
\]
which proves the second.
\end{proof}

In particular, if $\theta=0$ then $\bs{t}$ is basic. If moreover $R=0$ then the exterior covariant derivative with respect to the pullback connection $\nabla$ (see Proposition~\ref{interpretation.VH}) vanishes. 

\subsection*{The lifted connection}
The pair of connections ($D$, $\bs{D}$) allows to define a linear connection on the tangent bundle $\map{\tau_E}{TE}{E}$ (i.e. a linear connection on the manifold~$E$) by means of
\[
\nabla_U W=(\bs{D}_U\,T\pi(W))^\H+(D_U\,\kappa(W))^\V,\qquad U,W\in\vectorfields{E}.
\]
In more explicit terms 
\[
\nabla_U(X^\H+\sigma^\V)=(\bs{D}_UX)^\H+(D_U\sigma)^\V,
\]
for $U\in\vectorfields{E}$, $X\in\sec{\pi^*TM}$ and $\sigma\in\sec{\pi^*E}$. Obviously this connection carries the same information as both connections. This is the kind of connection defined in~\cite{Byrnes} for the case $E=TM$ where one can choose $\bs{D}=D$, and it is related to the connection in~\cite{Massa.Pagani}. 

\subsection*{Integral sections}
An integral section of an Ehresmann connection on $\map{\pi}{E}{M}$ is a section $\alpha\in\sec{E}$ such that
\begin{equation}
T\alpha(v)=\xi^\H(\alpha(\tau_M(v)),v)\qquad \forall v\in TM.
\end{equation}
Equivalently, $\alpha$ is an integral section if for any vector field $X\in\vectorfields{M}$ we have $T\alpha\circ X=X^\H\circ \alpha$.

In local coordinates, a section $\alpha=\alpha^A(x)e_A$ is an integral section of the connection if and only if it satisfies the partial differential equation
\begin{equation}
\pd{\alpha^A}{x^i}(x^j)+\Gamma^A_i(x,\alpha(x))=0.
\end{equation}
The first order integrability conditions for this system of partial differential equations impose the vanishing of the curvature on the image of an integral section. When the curvature vanishes, Frobenius theorem ensures that there exists a local integral section trough any point of $E$.

\begin{proposition}
Every integral section $\alpha$ defines by pullback a linear connection ${}^\alpha\!D$ on the bundle $\map{\pi}{E}{M}$ whose covariant derivative is given by
\begin{equation}
{}^\alpha\!D_{X}\sigma=\kappa([X^\H,\sigma^\V])\circ\alpha,
\end{equation}
for $\sigma\in\sec{E}$ and $X\in\vectorfields{M}$. The coefficients of this connection are $\Gamma^A_{jB}(x,\alpha(x))$. The curvature tensor of ${}^\alpha\!D$ is given by ${}^\alpha{\curv}(X,Y)\sigma=[\curv(X^\H,Y^\H)\sigma]\circ\alpha$.
\end{proposition}
\begin{proof}
The pullback of the linear connection $D$ by the map $\map{\alpha}{M}{E}$ defines a linear connection ${}^\alpha\!D$ on the bundle $\alpha^*(\pi^*E)\to M$. But this bundle is canonically isomorphic to $E\to M$, because $\pi\circ\alpha=\id_M$. The formula for the covariant derivative follows by noticing that any tangent vector to the image of $\alpha$ is a horizontal vector field. In other words, from $T\alpha\circ X=X^\H\circ\alpha$ it follows ${}^\alpha\!D_{X}\sigma={}^\alpha\!D_X(\sigma\circ\pi\circ\alpha)=D_{T\alpha\circ X}\sigma=D_{X^\H\circ\alpha}\sigma=\kappa([X^\H,\sigma^\V])\circ\alpha$.  The formula for the curvature follows from the structure equations for the curvature (see~\cite{Poor}, paragraph 2.65).
\end{proof}

\section{Linearization of a connection on an affine bundle}
\label{linearization.affine}

In many situations the bundle on which a nonlinear connection is defined is not a vector bundle but an affine bundle $\map{{\pi}}{A}{M}$ modeled on a vector bundle $\map{\bs{{\pi}}}{\bs{V}}{M}$. In such case the fundamental sequence is
\[
\seq 0-> \pi^*\bs{V} -\xi^\V-> TA -p_A-> \pi^*TM ->0,
\]
and the above construction is not readily applicable. We will extend the linearization process to the affine case by a natural procedure of homogenization and restriction. 

\subsection*{The linear span of an affine space}
It is convenient to think on an affine space as an affine hyperplane embedded in a bigger vector space as in~\cite{AffineInVector}. Let $\cala$ be a real finite dimensional affine space modeled on a vector space $\calv$. There exists a canonically defined vector space $\tilde{\cala}$, known as the linear span of $\cala$, endowed with a foliation by affine hyperplanes  $\cala_\lambda$ indexed by real numbers $\lambda\in\Real$, such that $\cala$ is canonically identified with $\cala_1$ and $\calv$ is canonically identified with $\cala_0$ (which is a vector subspace of $\tilde{\cala}$). In other words, there exists a short exact sequence of vector spaces 
\begin{equation}
\label{sequence.affine}
\seq 0-> \calv -{i_0}-> \tilde{\cala} -{j}-> \Real ->0, 
\end{equation}
and the hyperplanes $\cala_\lambda$ are the level sets of the function $j$, i.e. $\cala_\lambda=j^{-1}(\lambda)$ for $\lambda\in\Real$.

The above sequence can be constructed as follows. The vector space $\cala^\dag=\operatorname{Aff}(\cala,{\Real})$ of affine functions defined on $\cala$ is called the (extended) affine dual of $\cala$. It fits in an exact sequence of vector spaces
\begin{equation}
\label{sequence.affine.dual}
\seq 0-> \Real -k-> \cala^\dag -l-> \calv^* ->0,
\end{equation}
where the map $k$ associates to the real number $\lambda$ the constant function with value $\lambda$, and the map $l$ associates to an affine function its linear part. Then the dual vector space $\tilde{\cala}=(\cala^\dag)^*$ satisfies the conditions we have imposed, and the sequence~\eqref{sequence.affine} is just the dual of~\eqref{sequence.affine.dual}. The immersion $\map{i_1}{\cala}{\tilde\cala}$ is given by evaluation, $i_1(a)({\varphi})={\varphi}(a)$ for all ${\varphi}\in \cala^\dag$. The map $i_1$ is affine and the associated lineal map is the canonical immersion $\map{i_0}{\calv}{\tilde{\cala}}$. 

Conversely, it is easy to see that if $\seq 0->\calv-i->\calb-j->\Real->0$ is an exact sequence of vector spaces then $\cala=j^{-1}(1)$ is an affine subspace modeled on $i(\calv)=j^{-1}(0)$, and $\calb$ is canonically isomorphic to $\tilde{\cala}$.

Every point in $\tilde{\cala}-i_0(\calv)$ can be written in the form $z=\lambda i_1(a)$ for $\lambda\in\Real$ and $a\in\cala$ uniquely defined. In fact $\lambda=j(z)$ and hence $i_1(a)=z/j(z)$. This allows to extend any function $f$ defined on $\cala$ to a homogeneous function $\tilde{f}$ defined on $\tilde{\cala}-i_0(\calv)$ by $\tilde{f}(\lambda i_1(a))=\lambda f(a)$. 

Taking an affine frame $(O,\{e_1,\ldots,e_n\})$ on $\cala$ we have a local basis $\{\tilde{e}_0,\tilde{e}_1,\ldots,\tilde{e}_n\}$ of $\tilde{\cala}$ given by $\tilde{e}_0=i_1(O)$ and $\tilde{e}_i=i_0(e_i)$ for $i=1,\ldots,n$. Such a basis will be denoted simply by $\{e_0,e_1,\ldots,e_n\}$ and the context will make clear the meaning. As a consequence, the affine coordinates $(y^1,\ldots,y^n)$ on $\cala$ associated to that frame, define linear coordinates $(z^0,z^1,\ldots,z^n)$ on $\tilde{A}$. The maps $i_1$ and $i_0$ have the expression $i_1(y^1,\ldots,y^n)=(1,y^1,\ldots,y^n)$ and $i_0(z^1,\ldots,z^n)=(0,z^1,\ldots,z^n)$. The map $j$ has the expression $j(z^0,z^1,\ldots,z^n)=z^0$. The affine space $\cala$ is identified with the hyperplane $z^0=1$ and the vector space $\calv$ with the hyperplane $z^0=0$. More generally, the equation for the hyperplane $\cala_\lambda\subset\tilde{A}$ is $z^0=\lambda$. 

\begin{remark}
It is important to notice that the first element of the dual basis, $e^0\in \cala^\dag=\tilde{A}^*$, is but the affine function $j$ and hence it is canonically defined independently of the choice of frame. 
\end{remark}

\subsection*{Linearization of a connection}
The above constructions can be applied fiberwise to an affine bundle $\map{\pi}{A}{M}$ modeled on some vector bundle $\map{\bs{\pi}}{\bs V}{M}$. In this way we have defined a vector bundle $\map{\tilde{{\pi}}}{\tilde{A}}{M}$, with fibres $\tilde{A}_m=\widetilde{A_m}$ for $m\in M$, and canonical immersions $\map{i_1}{A}{\tilde{A}}$ and $\map{i_0}{\bs{V}}{\tilde{A}}$ over the identity in $M$. The canonical section on $\tilde{A}$ will be denoted by $\tilde{\cansec}$ and the canonical section on $A$ will be denoted just $\cansec$. We have the obvious property $i_1\circ\cansec=\tilde{\cansec}\circ i_1$. 

A local section $O\in \sec{A}$ and a local basis $\{e_1,\ldots,e_n\}$ of sections of $\bs{V}$ define a local basis $\{e_0,e_1,\ldots,e_n\}$ of sections of $\tilde{A}$ given by $e_0=i_1{\circ}O$ and $e_A\equiv i_0{\circ}e_A$. Taking local coordinates $(x^i)$ on the base $M$, we have associated coordinates $(x^i,y^A)$ on $A$, $(x^i,z^A)$ on $\bs{V}$ and $(x^i,z^0,z^A)$ on $\tilde{A}$, and the canonical immersions are $i_1(x^i,y^A)=(x^i,1,y^A)$ and $i_0(x^i,z^A)=(x^i,0,z^A)$. 

\medskip

\begin{remark}
In what follows in this paper we will consider $\bs{V}$ and $A$ as submanifolds of $\tilde{A}$ and consequently the maps $i_0$ and $i_1$ will be omitted. 
\end{remark}

\medskip

Given a nonlinear connection $\Hor{\subset}TA$ on $A$ we can define a nonlinear connection $\widetilde{\Hor}$ on $\tilde{A}-\bs{V}$ as the unique homogeneous connection on $\tilde{A}$ whose restriction to $A\subset\tilde{A}$ coincides with $\Hor$. The details of this construction are as follows. Every element $z\in \tilde{A}-\bs{V}$ is of the form $z={\lambda}a$ for ${\lambda}\in {\Real}-\{0\}$ and $a\in A_{\tilde{{\pi}}(z)}$. For $\lambda\in\Real$ we denote by ${\mu}_{\lambda}$ the multiplication map $\map{{\mu}_{\lambda}}{\tilde{A}}{\tilde{A}}$ given by ${\mu}_{\lambda}(z)={\lambda}z$. Then we define the nonlinear connection $\widetilde{\Hor}$ in $\tilde{A}-\bs{V}$ by the horizontal lift
\begin{equation}
\xi^{\tilde{\H}}({\lambda}a,v)=T{\mu}_{\lambda}\bigl(\hlift(a,v)\bigr).
\end{equation}
It is clear that $\xi^{\tilde{\H}}$ coincides with $\xi^\H$ on the points in $A$ and that the extended connection $\xi^{\tilde{\H}}$ is homogeneous, 
\begin{equation}
\xi^{\tilde{\H}}({\lambda}z,v)=T{\mu}_{\lambda}\bigl(\xi^{\tilde{\H}}(z,v)\bigr),\qquad\quad {\lambda}\in {\Real}-\{0\},
\end{equation}
but non linear, in general, because it is not defined on the submanifold $\bs{V}$. It follows that $[X^\H,{\Delta}]=0$ for any vector field $X\in \vectorfields{M}$, where $\Delta$ is the Liouville vector field on $\tilde{A}$

Moreover, this connection has the important property that $\widetilde{\Hor}$ is tangent to the foliation by hyperplanes, that is, for every $z\in \tilde{A}-\bs{V}$ we have that $\widetilde{\Hor}_z$ is tangent to the hyperplane $A_{j(z)}$. This is due to the property $\mu_{\lambda}(A_\rho)=A_{\lambda\rho}$ for $\rho\neq0$, and hence $T\mu_\lambda(TA)=TA_{\lambda}$ for $\rho=1$.

In local coordinates $(x^i,y^A)$ on $A$ and $(x^i,z^0,z^A)$ on $\tilde{A}$, if the expression of the horizontal lift $H_i$ is 
\[
H_i\at{(x,y)}=\pd{}{x^i}\at{(x,y)}-{\Gamma}^A_i(x,y)\pd{}{y^A}\at{(x,y)},
\]
then the expression of the horizontal lift $\tilde{H}_i$ is 
\begin{equation}
\tilde{H}_i\at{(x,z^0,z)}=\pd{}{x^i}\at{(x,z^0,z)}-z^0{\Gamma}^A_i(x,z/z^0)\pd{}{z^A}\at{(x,z^0,z)}.
\end{equation}
In other words the connection coefficients are 
\begin{equation}
\label{coefficients.homogeneous}
\begin{aligned}
\tilde{{\Gamma}}^0_i(x,z^0,z)&=0\\
\tilde{{\Gamma}}^A_i(x,z^0,z)&=z^0{\Gamma}^A_i(x,z/z_0).
\end{aligned}
\end{equation}

The homogeneous nonlinear connection can be linearized, as explained in Section~\ref{linearization}, obtaining a covariant derivative $\tilde{D}$ on the bundle $\tilde{\pi}^*\tilde{A}\to\tilde{A}-\bs{V}$. Finally we can restrict $\tilde{D}$ to the submanifold $A\subset\tilde{A}-\bs{V}$ (in other words we pullback the connection by the inclusion $\map{i_1}{A}{\tilde{A}}$) and we get a covariant derivative $D$ on the bundle $\pi^*\tilde{A}\to A$ which is the linearized connection that we are looking for.

\begin{definition}
The linearization of a nonlinear connection on an affine bundle $\map{\pi}{A}{M}$ is the restriction to $A$ of the linearization of the homogeneous extension of the connection to $\tilde{A}-\bs{V}$.
\end{definition}

It is important to notice that, as $A$ is away from the singular set $\bs{V}$, the above connection is smooth on the whole manifold $A$.

The fact that the connection that we really linearize is homogeneous implies that the tension vanishes. Thus we have the following result.

\begin{proposition}
\label{derivative.of.cansec}
The nonlinear connection can be recovered from the linearization via the expression
\begin{equation}
D_U\cansec={\kappa}(U).
\end{equation}
\end{proposition}
\begin{proof}
We prove that the corresponding relation holds true for the homogeneous connection. 

For a horizontal vector field $U=X^\H$, in view of Proposition~\ref{homogeneous}, we have $\tilde{D}_{X^\H}\tilde{\cansec}=0$, because $\tilde{D}$ is the linearization of a homogeneous connection. For a vertical lift $U=\alpha^\V$ we also know that $\tilde{D}_{\alpha^\V}\tilde{\cansec}=\alpha$. Thus $\tilde{D}_{X^\H+\alpha^\V}\tilde{\cansec}=\alpha=\tilde{\kappa}(\alpha^\V)=\tilde{\kappa}(X^\H+\alpha^\V)$. By restricting to points in $A\subset\tilde{A}$ we get the result. \end{proof}

\subsection*{An explicit expression for $D$}
The linearized connection $D$ can be obtained from the definition, by homogenization and restriction, 
\[
(D_U\sigma)(a)
\mathrel{\mathop{=}\limits^{\mathrm{def}}}
(\tilde{D}_{\tilde{U}}\sigma)(a)
=\tilde{\kappa}\Bigl([\tilde{P}_{\H}(\tilde{U}),\sigma^\V]+T\sigma(\tilde{P}_\V(\tilde{U}))\Bigr)(a),
\qquad a\in A,
\]
where $\tilde{U}$ is any extension of $U\in\vectorfields{A}$ to a vector field on $\tilde{A}-\bs{V}$ (for instance the homogeneous extension). While this expression is well-defined in $\tilde{A}$, and the second term is just $\kappa(T\sigma(P_V(U)))(a)$, we must notice that the first term is not directly computable in terms of vector fields tangent to $A$. Essentially, the problem is that a vertical lift is not tangent to $A$. This is clear, for instance, for the vertical lift of the canonical section  $\tilde{\cansec}$ of $\tilde{A}$, for which $\tilde{\cansec}^\V=\Delta$ which is not tangent to the hyperplane foliation but tangent to the transversal rays $R_z=\set{\lambda z}{\lambda\in\Real}$.

However, once fixed a point $a\in A$, and setting $m=\pi(a)$, every vector $z\in\tilde{A}_m$ can be written as a sum $z=\pai{e^0}{z}a+(z-\pai{e^0}{z}a)$. The vector $\bar{z}=z-\pai{e^0}{z}a$ is in $\bs{V}_m$ because $j(\bar{z})=\pai{e^0}{\bar{z}}=\pai{e^0}{z-\pai{e^0}{z}p}=0$. Thus the fibre $\tilde{A}_m$ decomposes as a direct sum $\tilde{A}_m=\langle a \rangle\oplus \bs{V}_m$. As a consequence, the $\cinfty{A}$-module $\sec{\pi^*\tilde{A}}$ decomposes as a direct sum $\sec{\pi^*\tilde{A}}=\langle\cansec\rangle\oplus\sec{\pi^*\bs{V}}$: every section $\sigma\in\sec{\pi^*\tilde{A}}$ can be written in a unique way as a sum
\begin{equation}
\label{canonical.decomposition}
\sigma=\pai{e^0}{\sigma}\cansec+\bar{\sigma},\qquad\bar{\sigma}=\sigma-\pai{e^0}{\sigma}\cansec\in\sec{\pi^*\bs{V}}.
\end{equation}

Taking this decompositions into account we have the following explicit expression for the covariant derivative.

\begin{proposition}
\label{linearization.explicit}
The covariant derivative $D$ can be expressed in terms of vector fields on $A$ in the form 
\begin{equation}
\label{linearization.affine.explicit}
D_U\sigma=\kappa\Bigl([P_\H(U),\bar{\sigma}^\V]+T\sigma(P_\V(U))\Bigr)+P_\H(U)(\pai{e^0}{\sigma})\cansec,
\end{equation}
for every $U\in\vectorfields{A}$ and every $\sigma\in\sec{\pi^*\tilde{A}}$.
\end{proposition}
\begin{proof}
We prove a similar expression for the linearization $\tilde{D}$ of the homogeneous extension and later we restrict to $A$. Obviously, for a given section $\sigma\in\sec{\tilde{\pi}^*\tilde{A}}$ we can always write it as $\sigma=\pai{e^0}{\sigma}\tilde{\cansec}+\bar{\sigma}$, with $\bar{\sigma}=\sigma-\pai{e^0}{\sigma}\tilde{\cansec}$. At points in $A$ this is just the decomposition~\eqref{canonical.decomposition}. However, notice that $\bar{\sigma}$ takes values on $\bs{V}$ only at points in $A$.

The result is obvious for a vertical vector field $U$ on $A$. For a horizontal lift $U=X^{\tilde{H}}$ of a vector field $X\in\vectorfields{M}$, using the decomposition~\eqref{canonical.decomposition}, we have 
\begin{align*}
\tilde{D}_{X^{\tilde{\H}}}\sigma
&=\tilde{\kappa}([X^{\tilde{\H}},\sigma^\V])\\
&=\tilde{\kappa}([X^{\tilde{\H}},\bar{\sigma}^\V+\pai{e^0}{\sigma}\tilde{\cansec}^\V])\\
&=\tilde{\kappa}([X^{\tilde{\H}},\bar{\sigma}^\V]+[X^{\tilde{\H}},\pai{e^0}{\sigma}\Delta])\\
&=\tilde{\kappa}([X^{\tilde{\H}},\bar{\sigma}^\V])+\kappa(X^{\tilde{\H}}(\pai{e^0}{\sigma})\Delta)\\
&=\tilde{\kappa}([X^{\tilde{\H}},\bar{\sigma}^\V])+X^{\tilde{\H}}(\pai{e^0}{\sigma})\tilde{\cansec},
\end{align*}
where we have taken into account that $X^{\tilde{\H}}$ is homogeneous, and hence $[X^{\tilde{\H}},\Delta]=0$, and that $\tilde{\kappa}(\Delta)=\tilde{\cansec}$.

For a general horizontal vector field $H$, we can write $H=\sum_I f_I X_I^{\tilde{\H}}$ for some finite family $\{X_I\}$ of vector fields on $M$ and some functions $f_I\in\cinfty{\tilde{A}}$. Then
\[
[H,\sigma^\V]=\sum_If_I\Bigl([X_I^{\tilde{\H}},\bar{\sigma}^\V]+X_I^{\tilde{\H}}(\pai{e^0}{\sigma})\Delta\Bigr)-
\sum_I \sigma^V(f_I)X_I^{\tilde{\H}}.
\]
When applying $\tilde{\kappa}$, the last term vanishes, and we obtain
\[
\tilde{\kappa}([H,\sigma^\V])=\sum_If_I\tilde{D}_{X_I^{\tilde{\H}}}\sigma=\tilde{D}_{H}\sigma.
\]
Therefore we have
\[
\tilde{D}_{\tilde{U}}\sigma=\tilde\kappa\Bigl([P_{\,\tilde\H}(\tilde{U}),\bar{\sigma}^\V]+T\sigma(P_{\,\tilde\V}(\tilde{U}))\Bigr)+P_{\,\tilde\H}(\tilde{U})(\pai{e^0}{\sigma})\tilde\cansec.
\]

Finally, restricting to points in $A$, every vector field appearing in the right hand side of the above expression  is tangent to $A\subset\tilde{A}$, and we get~\eqref{linearization.affine.explicit}.  
\end{proof}

\subsection*{Coordinate expressions}
Taking a local frame as indicated above and the associated coordinates, and considering the local basis of vector fields 
\begin{equation}
H_i=\pd{}{x^i}-\Gamma^A_i(x,y)\pd{}{y^A},
\qquad\qquad
V_A=\pd{}{y^A},
\end{equation}
from equations~\eqref{coefficients.homogeneous} a straightforward calculation shows that
\begin{equation}
\begin{aligned}
D_{H_i}e_0&=(\Gamma^B_i-y^A\Gamma^B_{iA})e_B,\qquad
&D_{V_B}e_0&=0,\qquad\\
D_{H_i}e_A&=\Gamma^B_{iA}e_B,&
D_{V_B}e_A&=0,
\end{aligned}
\end{equation}
where as before we are using the notation $\Gamma^A_{iB}=\partial \Gamma^A_i/\partial y^B$. The coefficient $\Gamma^A_i-y^B\Gamma^A_{iB}$ will be denoted $\Gamma^A_{i0}$.

Let $U\in\vectorfields{A}$ be a vector field with local expression $U=U^iH_i+U^AV_A$. For a section $\sigma\in\sec{\pi^*\tilde{A}}$ with coordinate expression $\sigma=\sigma^0e_0+\sigma^Ae_A$ the covariant derivative of $\sigma$ has the expression
\begin{equation}
D_U\sigma=U(\sigma^0)e_0+\{U^i[H_i(\sigma^A)+\Gamma^A_{i0}\sigma^0+\Gamma^A_{iB}\sigma^B]+U^CV_C(\sigma^A)\}e_A.
\end{equation}
In particular, for a section $\bs{\zeta}\in\sec{\pi^*\bs{V}}$ with coordinate expression $\bs{\zeta}=\zeta^Ae_A$ the covariant derivative of $\bs{\zeta}$ has the expression 
\begin{equation}
D_U\bs{\zeta}=\{U^i[H_i(\zeta^A)+\Gamma^A_{iB}\zeta^B]+U^CV_C(\zeta^A)\}e_A,
\end{equation}
and for a section $\eta\in\sec{\pi^*A}$ with coordinate expression $\eta=e_0+\eta^Ae_A$ the covariant derivative of $\sigma$ has the expression
\begin{equation}
\label{connections.coefficients.affine}
D_U\sigma=\{U^i[H_i(\eta^A)+\Gamma^A_{i0}+\Gamma^A_{iB}\eta^B]+U^CV_C(\eta^A)\}e_A.
\end{equation}

\subsection*{The linearized connection is affine}
The word \textit{affine connection} applies to two different types of connections. On one hand a linear connection on the tangent bundle $\map{\tau_M}{TM}{M}$ is said to be an affine connection on the base manifold $M$. On the other, a nonlinear connection on an affine bundle such that its connection coefficients are affine functions is also said to be an affine connection on that affine bundle. It is in this second sense that we use the word affine in this paper.

An affine connection on an affine bundle $\map{\pi}{A}{M}$ is characterized by the existence of a covariant derivative $\bs{\nabla}$ on the vector bundle $\map{\bs{\pi}}{\bs{V}}{M}$ and an operator $\map{\nabla}{\vectorfields{M}\times\sec{A}}{\sec{\bs{V}}}$, which is $\cinfty{M}$-linear in the first argument and satisfies 
\[
\nabla_X(\sigma+\bs{\zeta})=\nabla_X\sigma+\bs{\nabla}_X\bs{\zeta},\qquad \forall \sigma\in\sec{A},\bs{\zeta}\in\sec{\bs{V}}.
\]
A (standard) covariant derivative $\tilde{\nabla}$ on $\tilde{A}$ can be defined such that $\tilde{\nabla}_X$ coincides with $\nabla_X$ on sections taking values in $A{\subset}\tilde{A}$. It can be shown that a covariant derivative $\bar{\nabla}$ on $\tilde{A}$ is the covariant derivative associated to an affine connection if and only if $\bar{\nabla}_Xe^0=0$ for all $X\in\vectorfields{M}$. See~\cite{con.affine.MeSaMa,con.affine.MeSa} for the details. 

Alternatively, with the terminology of this paper, an affine connection on an affine bundle $\map{\pi}{A}{M}$ is a nonlinear connection on $A$ such that its homogeneous extension to $\tilde{A}$ happens to be linear (i.e. smooth on $\tilde{A}$ including $\bs{V}$). 

In our case, given a nonlinear connection on $A$, the linearized connection $D$ is defined on the vector bundle $\map{\pr_1}{\pi^*\tilde{A}}{A}$. Equation~\eqref{connections.coefficients.affine} shows that the connection coefficients are affine functions indicating that $D$ is an affine connection on this bundle. We will give a direct proof of this fact.

\begin{proposition}
\label{De0=0}
The section $e^0\in\sec{\pi^*A^\dag}$ is parallel, 
\[
D_Ue^0=0,\qquad \forall U\in\vectorfields{A}.
\] 
The vector subbundle $\pi^*\bs{V}\subset\pi^*\tilde{A}$ and the affine subbundle $\pi^*A\subset\pi^*\tilde{A}$ are invariant under parallel transport, 
\begin{equation}
\begin{aligned}
&D_U\sec{\pi^*\bs{V}}\subset\sec{\pi^*\bs{V}},&&\forall U\in\vectorfields{A}\\
&D_U\sec{\pi^*A}\subset\sec{\pi^*\bs{V}},&&\forall U\in\vectorfields{A}.
\end{aligned}
\end{equation} 
\end{proposition}
\begin{proof}
An important fact to be noticed is that the differential 1-form $e^0\circ\tilde{\kappa}\in\sec{T^*\tilde{A}}$ is the exterior differential of the function $j$, that is, $e^0\circ\tilde{\kappa}=dj$.

We first prove that $\tilde{D}_Ue^0=0$ for all $U\in\vectorfields{\tilde{A}-\bs{V}}$. For every $\sigma\in\sec{\pi^*\tilde{A}}$ we have
\begin{align*}
\pai{\tilde{D}_Ue^0}{\sigma}
&=U(\pai{e^0}{\sigma})-\pai{e^0}{\tilde{D}_U\sigma}\\
&=U(\pai{e^0}{\sigma})-\pai{dj}{[\tilde{P}_\H(U),\sigma^\V]}-\pai{dj}{T\sigma(\tilde{P}_\V(U))}.
\end{align*}
The second term in this sum is 
\[
\pai{dj}{[\tilde{P}_\H(U),\sigma^\V]}
=\tilde{P}_\H(U)(\sigma^\V(j))-\sigma^\V(\tilde{P}_\H(U)(j))
=\tilde{P}_\H(U)(\pai{e^0}{\sigma}),
\]
where we have used $\tilde{P}_\H(U)(j)=0$ (the horizontal distribution is tangent to the hyperplane foliation) and $\sigma^\V(j)=\pai{dj}{\sigma^\V}=\pai{e^0\circ\kappa}{\sigma^\V}=\pai{e^0}{\kappa(\sigma^\V)}=\pai{e^0}{\sigma}$.
The third term in the above sum is 
\[
\pai{dj}{T\sigma(\tilde{P}_\V(U))}
=\tilde{P}_\V(U)(j\circ\sigma)
=\tilde{P}_\V(U)(\pai{e^0}{\sigma}).
\]
Thus 
\[
\pai{\tilde{D}_Ue^0}{\sigma}
=U(\pai{e^0}{\sigma})-\tilde{P}_\H(U)(\pai{e^0}{\sigma})-\tilde{P}_\V(U)(\pai{e^0}{\sigma})=0.
\]

From this property, restricting to points in $A\subset\tilde{A}$ we get $D_Ue^0=0$ for all $U\in\vectorfields{A}$. 

But also from that property it immediately follows that the covariant derivative of a section taking values on an hyperplane $A_\lambda$ is a section taking values in $\bs{V}$. The result follows by taking $\lambda=1$ for $A$ and $\lambda=0$ for $\bs{V}$, and, of course, restricting to points in $A$. 
\end{proof}

Due to the linearity in the first argument of the covariant derivative, the above properties can be equivalently stated as: $D_ue^0=0$  for all $u\in TA$; $D_u{\eta}\in \bs{V}_{{\pi}({\tau}_E(u))}$ for all ${\eta}\in \sec{{\pi}^*\bs{V}}$ and all $u\in TA$; and $D_u{\sigma}\in \bs{V}_{{\pi}({\tau}_E(u))}$ for all ${\sigma}\in\sec{{\pi}^*A}$ and all $u\in TA$.

\begin{remark}
Proposition~\ref{De0=0} establishes that the linearization of a nonlinear connection is an affine connection on the affine bundle $\map{\pr_1}{\pi^*A}{A}$. Moreover, when the initial nonlinear connection on $A$ is already affine with $(\nabla,\bs{\nabla})$ the associated covariant operators, then the restriction of $D$ to $\pi^*\bs{V}$ is the pullback by $\pi$ of the linear connection $\bs{\nabla}$, the restriction of $D$ to $\pi^*A$ is the pullback by $\pi$ of the affine operator $\nabla$, and the connection $\tilde{D}$ is the pullback by $\tilde{\pi}$ of the linear connection $\tilde{\nabla}$.
\end{remark}

\bigskip

To finish this section, it is worth to mention that, in the affine case, parallel transport is but the parallel transport defined by the linearization of the homogeneous connection. Parallel transport in $\tilde{A}$ preserves the foliation by hyperplanes: if $\gamma$ is a curve in $A$ and $z$ is a point in $A_\lambda$ then the parallel transport of $z$ along $\gamma$ remains in $A_\lambda$. Therefore the restriction of parallel transport to the affine subbundle $\pi^*A_\lambda\to A$ is an affine map for $\lambda\neq0$ and a linear map for $\lambda=0$. In particular, for $\lambda=1$ parallel transport in $\pi^*A\to A$ is an affine map whose associated linear map is the parallel transport in $\pi^*\bs{V}=\pi^*A_0\to A$.    

\section{Examples and applications}
\label{examples.applications}

We consider in this section the cases of a connection on a tangent bundle and on a first jet bundle, with special emphasis in connections defined by a \sode. We will show how the results in this paper cover those in our previous papers~\cite{linearizable.SODE, Bianchi, MikeFest} and we will give simpler proofs of some of such results. We postpone for the next section the analysis of connections on a cotangent bundle

\subsection*{Connections on the tangent bundle}
Inspired by our previous work on derivations of forms along the tangent bundle projection~\cite{deriv,deriv2}, for an Ehresmann connection on the tangent bundle we defined in~\cite{linearizable.SODE} the linear connection on $\map{\pr_1}{\tau_M^*TM}{TM}$ with covariant derivative
\begin{equation}
D_{U}Z=\kappa([P_\H(U),Z^\V])+T\tau_M([P_\V(U),Z^\H]),
\end{equation}
for $U\in\vectorfields{TM}$, $Z\in\sec{\tau_M^*TM}$. In that paper the notation $P_\V(U)^\downarrow$ was used instead of $\kappa(U)$. We now show that this connection is a particular case of the one given in this paper by showing that both expressions coincide whenever $E=TM$. 

We just have to show that $T\tau_M\circ[P_\V(U),Z^\H]=\kappa(TZ(P_\V(U))$, or equivalently, $T\tau_M([V,Z^\H])=\kappa(TZ(V))$ for any vertical vector field $V\in\vectorfields{TM}$, or equivalently that $S([V,Z^\H])=TZ(V)$, where $S$ is the vertical endomorphism of~$TTM$, $S=\xi^\V\circ p_{TM}$. In local coordinates $(x^i,v^i)$ on $TM$, if $V=V^i(x,v)\pd{}{v^i}$ and $Z=Z^i(x,v)\pd{}{x^i}$ we have 
\[
TZ(V)=V^i\pd{Z^j}{v^i}\pd{}{v^j},
\]
and on the other hand 
\[
[V,Z^\H]=V^k\pd{Z^i}{v^k}\left(\pd{}{v^i}-\Gamma^j_i\pd{}{v^j}\right)-Z^iV^k\pd{\Gamma^j_i}{v^k}\pd{}{v^j}-Z^iH_i(V^k)\pd{}{v^k},
\]
so that 
$
S([V,Z^\H])=V^i\pd{Z^j}{v^i}\pd{}{v^i}
$,
and both expressions coincide.

As a consequence of this fact, the connection $D$ itself is compatible (satisfies the condition~\eqref{natural.connection}) and we can use $\bs{D}=D$ in order to calculate Bianchi identities as expressed in~\eqref{Bianchi.identities}.

\medskip

The aim of~\cite{linearizable.SODE} was to characterize those \sode on $M$ which are linear in velocities or linear in all variables. We briefly review here part of this results. 

The canonical section in $\tau_M^*TM$ coincides with the total time derivative operator $\Ti$ with coordinate expression 
\[
\cansec_{TM}=\Ti=v^i\pd{}{x^i}.
\]

A \sode vector field  is a vectorfield $\Gamma\in\vectorfields{TM}$ which projects to $\Ti$. It has a local expression 
\begin{equation}
\Gamma=v^i\pd{}{x^i}+f^i(x,v)\pd{}{v^i}.
\end{equation}
It determines a nonlinear connection with connection coefficients
\begin{equation}
\Gamma^i_j=-\frac{1}{2}\pd{f^i}{v^j}.
\end{equation}
The details can be found in~\cite{Crampin.nonlinear.connection, Grifone}.

A key property of the connection defined by a \sode is that it is symmetric, in the sense that the connection coefficients satisfy the symmetry condition $\Gamma^i_{jk}=\Gamma^i_{kj}$, or in more intrinsic terms the torsion vanishes,
\[
[X,Y]=D_{X^\H}Y-D_{Y^\H}X, \qquad\forall X,Y\in\vectorfields{M}. 
\]
In fact, any nonlinear connection satisfying this property is the connection defined by a \sode.

\medskip

As a consequence we have the following fundamental property for \sode connections.
\begin{proposition}
If a distribution $\cald\subset\pi^*(TM)$ is parallel then it is basic and the base distribution is integrable.
\end{proposition}
\begin{proof}
As a consequence of Proposition~\ref{characterization.basic}, if $\cald$ is parallel then it is basic $\cald=\pi^*F$, for some subbundle $F\subset TM$. Moreover the bracket of two vector fields on $F$ is $[X,Y]=D_{X^\H}Y-D_{Y^\H}X$. As both $D_{X^\H}Y$ and $D_{Y^\H}X$ take values in $F$, we have that $[X,Y]$ takes also values on $\cald$. Thus $F$ is involutive and hence integrable.
\end{proof}

Another geometric object associated to the \sode $\Gamma$ is the (generalized) Jacobi endomorphism $\Phi$ which is a (1,1)-tensor field along the tangent bundle projection defined by $\Phi(X)=\kappa([\Gamma,X^\H])$ for every $X\in\sec{\tau_M^*TM}$. It has components 
\begin{equation}
\label{Jacobi.endomorphism}
\Phi^i_j=-\pd{f^i}{x^j}-\Gamma^i_k\Gamma^k_j-\Gamma(\Gamma^i_j).
\end{equation}
See~\cite{deriv2} for the definition and~\cite{IP.autonomo} for applications.

\begin{proposition}
Let $\Gamma\in\vectorfields{TM}$ be a \sode. There exists coordinates $(x^i)$ on $M$ such that the differential equation for the integral curves of $\Gamma$ is
\begin{itemize}
\item linear in velocities, i.e.
\[
\ddot{x}^i=A^i_j(x)v^j+b^i(x),
\]
if and only if the linearized connection $D$ is flat. The coefficients $A^i_j$ are constant if and only if the tension is parallel. 
\item linear in all variables, i.e. 
\[
\ddot{x}^i=A^i_jv^j+B^i_jx^j+c^i
\]
if and only if the linearized connection $D$ is flat and the Jacobi endomorphism $\Phi$ is parallel. 
\end{itemize}
\end{proposition}
The details of the proof are in~\cite{linearizable.SODE}

Further applications of the linear connection to the Inverse problem of Lagrangian Mechanics can be found in~\cite{Douglas, IP.autonomo}. Also application to the decoupling of second-order differential equations can be found in~\cite{decoupling}. The problem of decoupling (for non-autonomous \sode{}s) will be reviewed in the next subsection.

\subsection*{Connections on the first jet bundle}
The linearization of a connection was extended to the non-autonomous setting in~\cite{Bianchi}. The non-autonomous formalism is developed in the first jet bundle $\map{\pi=\tau_{10}}{J^1\tau}{M}$ to a bundle $\map{\tau}{M}{\Real}$. The bundle $J^1\tau$ is an affine bundle modeled on the vector bundle $\Ver(\tau)=\Ker(T\tau)\subset TM$, and therefore the theory developed in Section~\ref{linearization.affine} applies to this case. Indeed, with the notation in this paper the connection on the the bundle $\map{\pr_1}{\pi^*(TM)}{J^1\tau}$ defined in~\cite{Bianchi} reads
\begin{equation}
D_U\sigma=\kappa([P_\H(U),\bar{\sigma}^\V])  
+T\pi([P_\V(U),\sigma^\H]) 
+P_\H(U)(\langle\sigma,dt\rangle)\Ti.
\end{equation}
As in the autonomous case the term $T\pi([P_\V(U),\sigma^\H])$ can be easily seen to be equal to $\kappa(T\sigma(P_\V(U)))$, so that this connection coincides with the one given in Proposition~\ref{linearization.explicit}. 

In~\cite{Bianchi} the linear connection was defined by looking for an expression similar to the one in the  autonomous case~\cite{linearizable.SODE} and the extra term $P_\H(U)(\langle\sigma,dt\rangle)\Ti$ was introduced to satisfy all the requirements of being a connection, but there was some freedom in the choice. We have seen in this paper that this is the natural choice. 

We will denote by $t$ the canonical coordinate in $\Real$, and we will take local coordinates $(t,x^i)$ on $M$ adapted to the projection $\tau$. They induce coordinates $(t,x^i,v^i)$ in $J^1\pi$. In our local expressions we will only use this kind of coordinates.

On $J^1\tau$ the canonical section coincides with the non-autonomous total time derivative operator $\cansec_{J^1\tau}=\Ti$ with coordinate expression 
\[
\Ti=\pd{}{t}+v^i\pd{}{x^i}.
\] 
A non-automous \sode $\Gamma$ is a vector field $\Gamma\in\vectorfields{J^1\tau}$ which projects to the canonical section. Its integral curves are 1-jet prolongation of section of $\tau$ and has a local expression of the form
\[
\Gamma=\pd{}{t}+v^i\pd{}{x^i}+f^i(t,x,v)\pd{}{v^i}.
\]

A \sode $\Gamma$ defines a nonlinear connection on the affine bundle $\map{\pi}{J^1\tau}{M}$ with connection coefficients given by 
\begin{equation}
\label{connection.coefficients.non-autonomous}
\Gamma^i_0=\frac{1}{2}v^k\pd{f^i}{v^k}-f^i
\qquad\text{and}\qquad
\Gamma^i_j=-\frac{1}{2}\pd{f^i}{v^j}.
\end{equation}
An intrinsic definition in terms of the vertical endomorphism can be found in~\cite{Crampin.connection.non-autonomous}. 

To define the associated linear connection as indicated in the previous section we should find its homogeneous extension and then linearize it. However, alternatively, the homogeneous extension of such nonlinear connection can be directly obtained as the nonlinear connection defined by the autonomous \sode $\Gamma_h\in\vectorfields{TM-\Ver(\tau)}$ known as the homogeneous \sode-extension of $\Gamma$. In local coordinates $(t,x^i,w^0,w^i)$ in $TM$ the homogeneous \sode $\Gamma_h$ has the expression 
\begin{equation}
\label{homogeneous.sode}
\Gamma_h=w^0\pd{}{t}+w^i\pd{}{x^i}+(w^0)^2f^i\bigl(t,x^j,w^j/w^0\bigr)\pd{}{w^i},
\end{equation}
from where~\eqref{connection.coefficients.non-autonomous} follows immediately.

In more intrinsic terms, the homogeneous \sode $\Gamma_h$ can be defined as follows. For a point $v_0\in J^1\tau$ we set $t_0=\tau(\pi(v_0))$. Let $\map{\gamma}{\Real}{M}$ be the integral section of $\Gamma$ through the point $v_0$, so that $\Gamma\circ j^1\gamma=\frac{d}{dt}j^1\gamma$ and $j^1_{t_0}\gamma=v_0$. For $\lambda\in\Real-\{0\}$ we define the curve $\map{\eta}{\Real}{M}$ by $\eta(s)=\gamma(t_0+\lambda s)$. Then the value of $\Gamma_h\in\vectorfields{TM-\Ver(\tau)}$ at the point $z=\lambda v_0\in TM$ is defined by 
\[
\Gamma_h(z)=\frac{d}{ds}\left(\frac{d\eta}{ds}\right)\at{s=0}.
\]
In this way the integral curves of $\Gamma_h$ are reparameterization of the integral curves of $\Gamma$ with constant speed $\lambda$ in the variable~$t$.

The vector field $\Gamma_h$ is homogeneous but, in general, it is not defined on the submanifold $\Ver(\tau)\subset TM$. Therefore it is a homogeneous spray. In fact $\Gamma_h$ can be extended continuously to $\Ver(\tau)$ if and only if the functions $f^i(t,x,v)$ are polynomials of degree two in the variables $v$. As a consequence the tension of this connection vanishes.

The benefits of this approach are clear: being the restriction to the affine bundle $J^1\tau$ of an autonomous \sode connection, every property that we have seen for the autonomous case applies also to the non-autonomous case, with some simplifications due to the fact that $\Gamma_h$ is a homogeneous spray, and hence $\Gamma$ is the horizontal lift of the canonical section $\Gamma=\Ti^\H$.

Due to the fact that $\Gamma$ is horizontal, in the non-autonomous setting the Jacobi endomorphism is part of the curvature tensor
\begin{equation}
\Phi(X)=\kappa([\Gamma,X^\H])=R(\Ti,X),
\qquad X\in\sec{\pi^*\Ver(\tau)},
\end{equation}
and has the same coordinate expression~\eqref{Jacobi.endomorphism} as in the autonomous case (but now all functions depend also on the variable $t$). In principle, one can consider the action of $\Phi$ on the full bundle $\pi^*TM$, but we have restricted to the vertical bundle because for $X=\Ti$ we clearly have $\kappa([\Gamma,\Ti^\H])=\kappa([\Gamma,\Gamma])=0$, which carries no information.

As an application we consider the problem of decoupling a second-order differential equation. We say that a \sode $\Gamma$ is locally submersive if in a neighborhood of every point in $M$ there exists a coordinate system $(t,x^i,x^\alpha)$ on $M$ such that the differential equation for the integral curves of $\Gamma$ has the form 
\[
\left\{\begin{aligned}
&\ddot{x}^i=f^i(t,x^j,\dot{x}^j),&& i,j=1,\ldots,d\\
&\ddot{x}^\alpha=f^\alpha(t,x^j,x^\beta,\dot{x}^j,\dot{x}^\beta),\qquad&& \alpha,\beta=d+1,\ldots,n.
\end{aligned}\right.
\]
We will say that $\Gamma$ locally decouples if furthermore such equations are of the form
\[
\left\{\begin{aligned}
&\ddot{x}^i=f^i(t,x^j,\dot{x}^j),&& i,j=1,\ldots,d\\
&\ddot{x}^\alpha=f^\alpha(t,x^\beta,\dot{x}^\beta),\qquad&& \beta=d+1,\ldots,n.
\end{aligned}\right.
\]

Notice that our coordinate systems are always assumed to be adapted to the fibration $\map{\tau}{M}{\Real}$. That is, we only consider coordinate transformations of the form $\bar{t}=t$, $\bar{x}^i=\psi^i(t,x^j)$.
 
We say that a distribution $\cald$ along $\pi$ is vertical if $\cald\subset\pi^*(\Ver(\tau))$. 

\begin{theorem}
Let $\Gamma\in\vectorfields{J^1\tau}$ be a \sode. Then  
\begin{itemize}
\item $\Gamma$ is locally submersive if and only if there exists a non-trivial $\Phi$-invariant parallel vertical distribution.
\item $\Gamma$ locally decouples if and only if there exists a pair of non-trivial $\Phi$-invariant parallel complementary vertical distributions.
\end{itemize}
\end{theorem}
\begin{proof}
If a vertical distribution $\cald$ is parallel then it is basic and the base distribution is integrable. We can find coordinates $(t,x^i,x^\alpha)$ on $M$ such that the annihilator of $\cald$ in $\pi^*\Ver(\tau)$ is generated by $\{dx^\alpha\}$. From $\pai{dx^\alpha}{D_\Gamma(\partial/\partial x^j)}=0$ it follows that $\Gamma^i_{\alpha}=0$, so that $\partial{f^i}/\partial{v^\alpha}=0$. The $\Phi$-invariance condition implies that 
\[
0=\Phi^i_\alpha=-\pd{f^i}{x^\alpha}-\Gamma^i_j\Gamma^j_\alpha-\Gamma^i_\beta\Gamma^\beta_\alpha-\Gamma(\Gamma^i_\alpha)=-\pd{f^i}{x^\alpha}.
\] 
Therefore, in this coordinates $f^i$ depends only on the coordinates $(t,x^j,v^j)$ but not on $(x^\alpha,v^\alpha)$. This proves the first statement. 

For the second, since the distributions are complementary on $\Ver(\tau)$, i.e. $\cald_1\oplus\cald_2=\pi^*(\Ver(\tau))$, we can find coordinates $(t,x^i,x^\alpha)$ on $M$ such that the annihilator of $\cald_1$ is generated by $\{dx^\alpha\}$ and the annihilator of $\cald_2$ is generated by $\{dx^i\}$. In this coordinates we have $\Gamma^i_\alpha=\Gamma^\alpha_i=0$ and $\Phi^i_\alpha=\Phi^\alpha_i=0$  from where it follows that $\partial{f^i}/\partial{v^\alpha}=\partial{f^\alpha}/\partial{v^i}=0$ and $\partial{f^i}/\partial{x^\alpha}=\partial{f^\alpha}/\partial{x^i}=0$. Hence $\Gamma$ decouples.
\end{proof}

The situation where a \sode decouples completely as a system of $n$ independent second-order 1-dimensional differential equations was studied in~\cite{decoupling.nonautonomous}. Generalized submersiveness was studied in~\cite{generalized.submersive}. Further applications to the Inverse Problem of the Lagrangian Mechanics can be found in~\cite{IP.nonautonomous}.

\section{Connections on the cotangent bundle}
\label{cotangent}

The theory developed in this paper allows to linearize connections on any vector bundle. In what respect to Mechanics the most relevant vector bundle is the cotangent bundle to a manifold. We show here how our results specialize to this case and we explore some further properties of the linear connection. In particular it will be shown that the geometry of the linearized connection can be an important tool in Hamiltonian Mechanics.

For the cotangent bundle $\map{\pi_M}{T^*M}{M}$ to a manifold $M$ the fundamental sequence is 
\[
\seq 0->\pi_M^*{T^*M} -\xi^\V-> T{T^*M} -p_M-> \pi_M^*{TM} ->0,
\]
where $p_M=(\tau_{T^*M},T\pi_M)$. Given a nonlinear connection on the cotangent bundle every vector field $U\in\vectorfields{T^*M}$ can be written in a unique way as a sum $U=X^\H+\eta^\V$, for $X$ a vector field along $\pi_M$ and $\alpha$ a differential 1-form along $\pi_M$ (i.e. $X\in\sec{\pi_M^*TM}$ and $\eta\in\sec{\pi_M^*T^*M}$).

In local coordinates $(x^i,p_i)$ on the cotangent bundle, the horizontal lift $H_i$ of a coordinate vector
field $\partial/\partial x^i$ has the expression
\begin{equation}
H_i=\pd{}{x^i}-\Gamma_{ij}(x,p)\pd{}{p_j}.
\end{equation}
The connection coefficients of the linearized connection are 
\begin{equation}
\Gamma^k_{ij}=\pd{\Gamma_{ij}}{p_k},
\end{equation}
and the coordinate expression of the covariant derivative is given by
\begin{equation}
\begin{aligned}
D_{X^\H}\alpha&=X^i[H_i(\alpha_j)+\alpha_k\Gamma^k_{ij}]dx^j
\\
D_{\eta^\V}\alpha&=\eta_i\pd{\alpha_j}{p_i}dx^j.
\end{aligned}
\end{equation}

As in the previous cases we can give an alternative expression of the linearized connection in terms of Lie derivatives provided that we accept some identifications. Let $\alpha$ be a 1-form along $\pi_M$. We can identify $\alpha$ with a $\pi_M$-semibasic 1-form $\bar{\alpha}$. The Lie derivative of $\bar{\alpha}$ with respect to a vertical vector field $V$ is again a $\pi_M$-semibasic 1-form and hence $\call_V\bar{\alpha}$ can be identified with a 1-form along $\pi_M$. If we write the resulting 1-form along $\pi_M$ simply as $\call_V\alpha$ then we have $\kappa(T\alpha(V))=\call_V\alpha$, which can be easily proved in coordinates. Therefore, with the above identifications, the linearized connection has the expression 
\begin{equation}
\label{linearization.cotangent}
D_U\alpha=\kappa([P_\H(U),\alpha^\V])+\call_{P_\V(U)}\alpha.
\end{equation}

The dual connection to $D$ is a connection on $\pi_M^*TM\to T^*M$ that satisfies the compatibility condition~\eqref{natural.connection}. Therefore, considering the connection $\bs{D}$ as the dual connection of $D$ we have that the expressions~\eqref{brackets} and~\eqref{Bianchi.identities} hold with $\bs{D}=D$. In particular,
\begin{equation}
[X^\H,\alpha^\V]=(D_{X^\H}\alpha)^\V-(D_{\alpha^\V}X)^\H,
\end{equation}
for $X$ a vector field along $\pi_M$ and $\alpha$ a 1-form along $\pi_M$. In fact this expression determines the linear connection~$D$, together with its dual. Obviously the lifted connection $\nabla$ on the manifold $T^*M$ can be defined without any auxiliary connection. 

The torsion of the linear connection (i.e. of the dual to the linearized connection) has the expression
\[
T=-\frac{1}{2}(\Gamma_{ij}^k-\Gamma_{ji}^k)dx^i\wedge dx^j\otimes\pd{}{x^k}.
\]

\medskip

For a function $f\in\cinfty{T^*M}$ we will denote by $d^\H f$ the 1-form along $\pi_M$ defined by $\pai{d^\H f}{X}=X^\H(f)$. Similarly we will denote by $d^\V f$ the vector field along $\pi_M$ defined by $\pai{\alpha}{d^\V f}=\alpha^\V(f)$. In terms of this objects the differential of a function $f\in\cinfty{T^*M}$ can be obtained from 
\begin{equation}
\label{decomposition.df}
\pai{df}{X^\H+\alpha^\V}=\pai{d^\H f}{X}+\pai{\alpha}{d^\V f},
\end{equation}
for $X\in\sec{\pi_M^*TM}$ and $\alpha\in\sec{\pi_M^*T^*M}$.

\medskip

\subsection*{Relation with symplectic geometry}

The geometry of the cotangent bundle is dominated by the canonical symplectic structure. We can expect that our results will be related to such structure.

The canonical section on $T^*M$ can be identified with the Liouville 1-form $\cansec_{T^*M}=\theta_0$, by means of  the identification of a $\pi_M$-semibasic 1-form with a section along the projection $\pi_M$. This is clear in coordinates since both have the same coordinate expression $p_idx^i$.

The tension is a 2-covariant tensor along $\pi_M$, $\bs{t}(X,Y)=-\pai{D_{X^\H}\theta_0}{Y}$, which in local coordinates reads
\begin{equation}
\label{Dh.theta}
\bs{t}=-D^\H\theta_0=(\Gamma_{ij}-\Gamma_{ij}^kp_k)dx^i\otimes dx^j.
\end{equation}

Proposition~\ref{homogeneous} can be rephrased as follows.
\begin{proposition}
A connection on the cotangent bundle $\map{\pi_M}{T^*M}{M}$ is homogeneous if and only if $D_{X^\H}\theta_0=0$ for all $X\in\sec{\pi_M^*TM}$. 
\end{proposition}

We consider the canonical symplectic structure $\omega_0=-d\theta_0$ on the cotangent bundle. As it is well known the vertical subbundle is a Lagrangian subbundle for the canonical symplectic form $\omega_0$, so that $\omega_0(\alpha^\V,\beta^\V)=0$ for all $\alpha,\beta\in\sec{\pi_M^*T^*M}$. For a horizontal and a vertical vector fields it is easy to see that
\begin{equation}
\label{omega(H,V)}
\omega_0(X^\H,\alpha^\V)=\pai{\alpha}{X},
\end{equation}
for all $X\in\sec{\pi_M^*TM}$ and all $\alpha\in\sec{\pi_MT^*M}$. 

\begin{definition}
A nonlinear connection on $T^*M$ is said to be symmetric if the horizontal subbundle is a Lagrangian subbundle with respect to $\omega_0$, i.e.\ 
\[
\omega_0(X^\H,Y^\H)=0,\qquad \forall X,Y\in\sec{\pi_M^*TM}.
\]
\end{definition}

This terminology is motivated by the fact that the connection is symmetric if and only if the connection coefficients are symmetric $\Gamma_{ij}(x,p)=\Gamma_{ji}(x,p)$. Thus, $\sigma(X,Y)=\omega_0(X^\H,Y^\H)$ is a 2-form along $\pi_M$ measuring the symmetry of the nonlinear connection, and can be properly called the torsion of the nonlinear connection. It has the local expression $\sigma=\frac{1}{2}(\Gamma_{ij}-\Gamma_{ji})dx^i\wedge dx^j$. The vertical differential of the torsion of the nonlinear connection, up to a sign, is the torsion $T$ of the associated linear connection, 
\[
\pai{T(X,Y)}{\zeta}=-(D_{\zeta^\V}\sigma)(X,Y),\qquad\forall X,Y\in\sec{\pi_M^*TM}.
\]
Notice that for a symmetric nonlinear connection, the tension is a symmetric 2-covariant tensor.

\begin{proposition}
\label{prop.Hamiltonian.vector.field}
For a symmetric connection the Hamiltonian vector field associated to a function $f\in\cinfty{T^*M}$ can be decomposed in the form 
\begin{equation}
\label{Hamiltonian.vector.field}
X_f=(d^\V f)^\H-(d^\H f)^\V.
\end{equation}
Conversely, if this equation holds for every function $f\in\cinfty{T^*M}$, then the connection is symmetric. 
\end{proposition}
\begin{proof}
The vector field $X_f$ is the solution of the equation $i_{X_f}\omega_0=df$. We can write $X_f=Z^\H+\eta^\V$ for uniquely defined $Z\in\sec{\pi_M^*TM}$ and $\eta\in\sec{\pi_M^*T^*M}$. For any $Y\in\sec{\pi_M^*TM}$, using equation~\eqref{omega(H,V)} and taking into account that the horizontal subbundle is Lagrangian we get
\[
\omega_0(X_f,Y^\H)=\omega_0(\eta^\V,Y^\H)=-\pai{\eta}{Y}.
\]
On the other hand, from the definition of $d^\H f$ we have
\[
\pai{df}{Y^\H}=Y^\H f=\pai{d^\H f}{Y},
\] 
from where we get $\eta=-d^\H f$. 

Similarly, as the vertical subbundle is also Lagrangian, for every $\sigma\in\sec{\pi_M^*T^*M}$ we have
\[
\omega_0(X_f,\sigma^\V)=\omega_0(Z^\H,\sigma^\V)=\pai{\sigma}{Z},
\]
and on the other hand
\[
\pai{df}{\sigma^\V}=\sigma^\V f=\pai{\sigma}{d^\V f},
\] 
from where we get $Z=d^\V f$. We conclude that $X_f=(d^\V f)^\H-(d^\H f)^\V$.

Conversely, for every $X\in\vectorfields{M}$ we consider the function $f=\pai{\theta_0}{X}\in\cinfty{TM}$, which satisfies $d^\V f=X$. Then for every $Y\in\vectorfields{M}$ we have
\begin{align*}
\omega_0(X^\H,Y^\H)
&=\omega_0((d^\V f)^\H,Y^\H)\\
&=\omega_0(X_f-(d^\H f)^\V,Y^\H)\\
&=\omega_0(X_f,Y^\H)-\omega_0((d^\H f)^\V,Y^\H)\\
&=\pai{df}{Y^\H}-\pai{d^\H f}{Y}\\
&=0,
\end{align*}
showing that the horizontal distribution is Lagrangian.
\end{proof}

\begin{proposition}
\label{prop.poisson.bracket}
For a symmetric connection on the cotangent bundle, the Poisson bracket of two functions $f,g\in\cinfty{T^*M}$ can be written in the form  
\begin{equation}
\label{poisson.bracket}
\{f,g\}=\pai{d^\H f}{d^\V g}-\pai{d^\H g}{d^\V f}.
\end{equation}
Conversely, if this equation holds for all $f,g\in\cinfty{T^*M}$, then
the connection is symmetric.
\end{proposition}
\begin{proof}
For a symmetric connection we have 
\begin{align*}
\{f,g\}
&=\omega_0(X_f,X_g)\\
&=\omega_0((d^\V f)^\H-(d^\H f)^\V,(d^\V g)^\H-(d^\H g)^\V)\\
&=\omega_0((d^\V g)^\H,(d^\H f)^\V)-\omega_0((d^\V f)^\H,(d^\H g)^\V)\\
&=\pai{d^\H f}{d^\V g}-\pai{d^\H g}{d^\V f}.
\end{align*}

Conversely, if~\eqref{poisson.bracket} holds for every $f,g\in\cinfty{T^*M}$ then
\begin{align*}
\pai{df}{X_g}
&=\{f,g\}\\
&=\pai{d^\H f}{d^\V g}-\pai{d^\H g}{d^\V f}\\
&=\pai{df}{(d^\V g)^\H-(d^\H g)^\V)}.
\end{align*}
As this expression holds for every function $f$ we have that $X_g=(d^\V g)^\H-(d^\H g)^\V$ for every function $g$, and Proposition~\ref{prop.Hamiltonian.vector.field} implies that the connection is symmetric.
\end{proof}

We finally mention the following two simple properties valid for a symmetric nonlinear connection:
\begin{itemize}
\item The curvature satisfies the following identity
\[
\sum\pai{R(X,Y)}{Z}=0,
\]
where the sum is over cyclic permutations of $X,Y,Z\in\sec{\pi_M^*TM}$. Indeed, computing the expression of $d\omega_0(X^\H,Y^\H,Z^\H)=0$, for $X,Y,Z\in\vectorfields{M}$, by Cartan formula for the exterior differential we have
\begin{align*}
0=d\omega_0(X^\H,Y^\H,Z^\H)
&=\sum X^\H\omega_0(Y^\H,Z^\H)-\omega_0([X^\H,Y^\H],Z^\H)\\
&=-\sum \omega_0([X,Y]^\H,Z^\H)-\omega_0(R(X,Y)^\V,Z^\H)\\
&=\sum\pai{R(X,Y)}{Z}.
\end{align*}
\item The lifted linear connection $\nabla$ on the manifold $T^*M$ is a symplectic connection, $\nabla\omega_0=0$. This fact is easy to see in coordinates by taking into account that the symplectic form can be written $\omega_0=H^i\wedge V_i$, where $H^i=dx^i$ and $V_i=dp_i+\Gamma_{ji}dx^j$ are the elements of the dual basis of $\{H_i,V^i\}$.
\end{itemize}

\subsection*{Completely integrable systems and Hamilton-Jacobi theory}

The linearization of a connection on the cotangent bundle can be used to obtain information on some problems in Hamiltonian Mechanics. We will show here a very simple application in studying properties of completely integrable systems.

Consider a completely integrable Hamiltonian system on $T^*M$ with Hamiltonian function $H\in\cinfty{T^*M}$. Let $\{f_1,\ldots,f_n\}$ functionally independent constants of the motion for $H$ which are in involution. The Hamiltonian $H$ is functionally dependent of such functions, usually taken the first of them. We will moreover assume the following natural transversality condition $d^\V f_1\wedge \cdots \wedge d^\V f_n\neq0$. In coordinates this corresponds to the regularity of the matrix $[\partial f_i/\partial p_j]$ at every point in $T^*M$. The Hamiltonian vector fields $\{X_{f_1},\ldots,X_{f_n}\}$ span an $n$-dimensional distribution, which is horizontal in view of the imposed transversality condition, and hence a nonlinear connection on $T^*M$. In equivalent terms, $\Hor$ is the annihilator of the linear span of $\{df_1,\ldots,df_n\}$, so that $d^\H f_i=0$.  On the other hand the functions $\{f_1,\ldots,f_n\}$ are in involution, i.e. satisfy $\{f_i,f_j\}=0$, and hence Proposition~\ref{prop.poisson.bracket} implies that $\Hor$ is an integrable Lagrangian distribution. Thus $\Hor$ is a flat symmetric nonlinear connection on~$T^*M$.

\medskip

From a dual point of view, a solution of the Hamilton-Jacobi equation for $H$ is a closed 1-form $\alpha\in\sec{T^*M}$ such that $H$ is constant on the image of $\alpha$, i.e. $H(x,\alpha(x))=c$ (constant). Locally $\alpha=dS$ for some function $S$ and we get the standard form of the Hamilton-Jacobi equation $H(x,\pd{S}{x})=c$. 

\begin{proposition}
\label{HJ.integral.section}
Let $\Hor$ be a symmetric nonlinear connection on $T^*M$ such that $d^\H H=0$. Every integral section of $\Hor$ is a  solution of the Hamilton-Jacobi equation.
\end{proposition}
\begin{proof}
Let $\alpha$ be an integral section of $\Hor$. For every vector $v\in TM$, if $m=\tau_M(v)$ we have 
\[
\pai{d(H\circ\alpha)}{v}=v(H\circ\alpha)=T\alpha(v)H=\xi^\H(\alpha(m),v)H=\pai{d^\H H(\alpha(m))}{v}=0.
\]
Therefore $H\circ\alpha$ is locally constant. 

Moreover $\alpha$ is closed because the connection is symmetric. Indeed, for every $m\in M$ and every $v,w\in T_mM$ we have
\[
(\alpha^*\omega_0)_m(v,w)=\omega_0(T\alpha(v),T\alpha(w))
=\omega_0\bigl(\xi^\H(\alpha(m),v),\xi^\H(\alpha(m),w)\bigr)=0
\]
because $\Hor$ is Lagrangian. Therefore $\alpha^*\omega_0=-\pi_M^*(d\alpha)=0$ and hence $\alpha$ is closed.
\end{proof}

Notice that for a completely integrable system, in the above proposition we can interchange $H$ with any other of the functions $f_i$, which also satisfies $d^\H f_i=0$. 

\smallskip

A complete solution of the Hamilton-Jacobi equation is a foliation of $T^*M$ where every leaf is the image of a solution of the Hamilton-Jacobi equation. As it is well known if the Hamilton-Jacobi equation admits a complete integral then $H$ is completely integrable. The relation between the complete integral and the family of first integrals in involution $\{f_1,\ldots,f_n\}$ is as follows. Parameterizing (perhaps locally) the complete integral so that $\{\alpha_\lambda\}_{\lambda\in\Lambda}$, $\Lambda\subset\Real^n$, is a family of closed 1-forms we get a (local) diffeomorphism $\map{\psi}{M\times\Lambda}{T^*M}$. The second component of the inverse $\map{\psi^{-1}}{T^*M}{M\times\Real^n}$ defines a function $\map{F=(f_1,\ldots,f_n)}{T^*M}{\Real^n}$ whose components are in involution. By construction we have that $\Im\alpha_\lambda=F^{-1}(\lambda)$. See~\cite{HJ1,HJ2} for the details.

\medskip

We will next analyze the simplest possible situation, in which the horizontal distribution associated to a completely integrable geodesic Hamiltonian system defines (by linearization) a flat linear connection.

Let $g$ be a pseudo-Riemannian metric on $M$. We consider the inverse metric $G=g^{-1}$ on $T^*M$ and the associated Hamiltonian 
\[
H(p)=\frac{1}{2}G_{\pi_M(p)}(p,p),\qquad p\in T^*M.
\]
$H$ is said to be a geodesic Hamiltonian because the solutions of Hamilton equations, when projected to $M$, are the geodesics of the metric $g$. In local coordinates
\[
H(x^i,p_i)=\frac{1}{2}g^{ij}(x^k)p_ip_j.
\]

\begin{theorem}
Let $H$ be a completely integrable geodesic Hamiltonian with metric~$g$. If the associated linearized connection is flat then:
\begin{itemize}
\item The nonlinear connection is in fact linear and equal to the dual of the Levi-Civita connection $\nabla$ for the metric $g$.
\item The Hamilton-Jacobi equation is separable in any $\nabla$-affine coordinate system. 
\item In any $\nabla$-affine coordinate system all coordinates are ignorable.
\end{itemize}
\end{theorem}
\begin{proof}
Let $\Hor$ be the nonlinear connection defined by the complete integrable system and let $D$ be the linearized connection. If $D$ is flat there exists a flat linear connection $\nabla$ on $T^*M$ such that $D$ is the pullback of $\nabla$. Moreover $\nabla$ is symmetric because $D$ is symmetric. It follows that $M$ is an affine manifold.

In a system of $\nabla$-affine coordinates $(x^i)$ in $M$ the coefficients of the nonlinear connection are basic local functions, $\Gamma_{ij}=\Gamma_{ij}(x)$. On the other hand, as $dH$ is linearly dependent on $\{df_1,\ldots,df_n\}$ we have that $d^\H H=0$. In the above coordinates this expression reads $ \pd{H}{x^i}-\Gamma_{ij}(x)\pd{H}{p_j}=0$. Taking the derivative with respect to $p_k$ of this expression we get
\[
\tag{$*$}
\pd{^2H}{x^i\partial p_k}=\Gamma_{ij}g^{jk}.
\] 
The left hand side of this expression is linear in momenta while the right hand side is a basic function, so that both must vanish. The right hand side is $\Gamma_{ij}g^{jk}=0$ and hence $\Gamma_{ij}=0$. Therefore the connection $\Hor$ and the connection $\nabla$ (as a nonlinear connection) both have null connection coefficients in the same coordinates, so that they are equal. From the vanishing of the left hand side of \thetag{$*$} we have $\partial g^{jk}/\partial x^i=0$ and hence $\nabla G=0$. Therefore $\nabla$ is a symmetric and isometric connection and hence the dual of the Levi-Civita connection for $g$. This proves the first statement.

Let $\alpha$ be any integral section of $\Hor$. In $\nabla$-affine coordinates $\partial\alpha_i/\partial x^j=0$ because the connection coefficients vanishes. Therefore the general solution of the equation of integral sections is $\alpha_\lambda=\sum_i\lambda_i dx^i$ for $\lambda=(\lambda_1,\ldots,\lambda_n)\in\Real^n$. As a consequence of Proposition~\ref{HJ.integral.section} we have found a complete solution of the Hamilton-Jacobi equation in separated form ($S_\lambda=\sum_i\lambda_i x^i$ with $\alpha_\lambda=dS_\lambda$). 

Finally, in any $\nabla$-affine coordinate system the metric coefficients are constant, $g^{ij}(x)=a^{ij}$, so that $H=\frac{1}{2}a^{ij}p_ip_j$ does not depend on the coordinates $x^i$. Thus all them are ignorable coordinates. 
\end{proof}

In the above proof, to simplify the arguments, we have used some coordinate calculations. Let us indicate how this calculations can be done intrinsically. The first step in the proof shows that $D$ is the pullback of a flat symmetric connection~$\nabla$. From the equation $d^\H H=0$ taking covariant derivatives with respect to a vertical vector field and using that the linear connection is flat we can easily get $D_{X^\H} (d^\V H)=0$. Taking again covariant derivatives with respect to a vertical vector field we arrive at $D_{X^\H} G=0$, where one has to use that the linear connection is flat and the equality $D_{\alpha^\V}D_{\beta^\V}H=G(\alpha,\beta)$ for every $\alpha,\beta\in\sec{T^*M}$. Since $G$ is a basic tensor and $D$ is the pullback of $\nabla$ one has $\nabla G=0$ and hence $\nabla$ is the dual of the Levi-Civita connection for the metric $g$.

\section{Conclusions and outlook}
In this paper we have extended to general vector bundles the procedure of linearization of a nonlinear connection following the ideas in~\cite{linearizable.SODE}, for a tangent bundle, and~\cite{Bianchi}, for the case of a first jet bundle. We have shown that the construction of the linearized connection in the affine case can be considered as a particular case of the linearization on a vector bundle of the homogenization of the connection.  

We have given an interpretation of the parallel rule defined by such connection, and we have studied the curvature of the linear connection. In particular a convenient expression of the Bianchi identities is provided.

As a relevant example, we have studied nonlinear connections on the cotangent bundle and we have shown that our results can be very useful in studying problems in Hamiltonian Mechanics. The simple problem that we have considered shows that further research is needed to understand problems related to completely integrable systems, as it is for instance the separability of the Hamilton-Jacobi equation.

\smallskip

Finally, in~\cite{con.affine.MeSaMa,con.affine.MeSa} a more general notion of connection on Lie algebroids and more generally on anchored vector and affine bundles was considered. Let us briefly indicate how our theory can be extended also to this case. If $\map{\pi}{E}{M}$ is a vector bundle and $(\map{\tau}{A}{M},\rho)$ is an anchored vector bundle, we just have to take the correct `tangent space' $\calt^A_bE=\set{(a,V)\in A_{\pi(b)}\times T_bE}{\rho(a)=T\pi(V)}$ at every point $b\in E$ and generalize our setting by considering the exact sequence $
\seq 0-> \pi^*E -\xi^\V-> \calt^AE -\mathsf{p}-> \pi^*A ->0$ of vector bundles over $E$. Here the map $\map{\mathsf{p}}{\calt^AE}{\tau^*A}$ is $\mathsf{p}(a,V_b)=(b,a)$ and the vertical lifting is considered as a map $\map{\xi^\V}{\pi^*E}{\calt^AE}$ given by $\xi^\V(b,c)=(0_{\pi(b)},\frac{d}{dt}(b+tc)|_{t=0})$. See~\cite{jetoids.variational, jetoids.multisymplectic} for more information on these spaces and maps. A generalized $A$-connection on $E$ is just a splitting of this sequence. The details of this constructions will be given elsewhere.


\let\sep\newline
\def\sep{,\ \ignorespaces}



\begin{thebibliography}{99}

\bibitem{Marsden.nonholonomic.horizontal}
  \textsc{Bloch AM, Krishnaprasad PS, Marsden JE, and Murray R}\sep
 \textsl{Nonholonomic mechanical systems with symmetry}\sep
 Arch. Rat. Mech. An. \textbf{136} (1996) 21--99. 

\bibitem{Byrnes}
 \textsc{Byrnes GB}\sep
 \textsl{A complete set of Bianchi identities for tensor fields along the tangent bundle projection}\sep
 J. Pys. A: Math. Gen. \textbf{27} (2002) 6617--6632.  

\bibitem{decoupling.nonautonomous}
  \textsc{Cantrijn F, Sarlet W, Vandecasteele A and Mart{\'\i}nez E}\sep
  \textsl{Complete separability of time-dependent second-order ordinary 
          differential equations}\sep
  Acta Appl. Math. \textbf{42} (1996), no. 3, 309--334.

\bibitem{Cardin.Favretti}
 \textsc{Cardin M and Favretti M}\sep
 \textsl{On nonholonomic and vakonomic dynamics of mechanical systems with nonintegrable constraints}\sep
 Journal of Geometry and Physics \textbf{18} no. 4 (1996) 295-325

\bibitem{HJ1}
    \textsc{Cari\~nena JF, Gr\`acia X, Marmo G, Mart\'{\i}nez E, Mu\~noz-Lecanda MC and Rom\'an-Roy~N}\sep
 \textsl{Geometric Hamilton-Jacobi theory}\sep
    International Journal on Geometric Methods in Mathematical Physics \textbf{3} (2006), no.~7, 1417--1458

\bibitem{HJ2}
    \textsc{Cari\~nena JF, Gr\`acia X, Marmo G, Mart\'{\i}nez E, Mu\~noz-Lecanda MC and Rom\'an-Roy~N}\sep
 \textsl{Geometric Hamilton-Jacobi theory for nonholonomic dynamical systems}\sep
    International Journal on Geometric Methods in Mathematical Physics \textbf{7} (2010), no.~3, 431--454.

\bibitem{sections.along.maps} 
   \textsc{Cari\~nena JF, L\'opez C and Mart\'{\i}nez E}\sep
 \textsl{Sections along a map applied to higher-order Mechanics. Noether's Theorem}\sep
    Acta Applicandae Mathematicae \textbf{25} (1991) 127--151.

\bibitem{Crampin.nonlinear.connection} 
   \textsc{Crampin M}\sep
   \textsl{Tangent bundle geometry for Lagrangian dynamics}\sep
   J. Phys. A \textbf{16} (1983) 3755--3772.

\bibitem{Crampin.Berwald} 
  \textsc{Crampin M}\sep
  \textsl{Connections of Berwald type}\sep
  Publ. Math. Debrecen \textbf{57} (2000), no. 3-4, 455--473


\bibitem{Bianchi}
 \textsc{Crampin M, Mart\'{\i}nez E and Sarlet W}\sep
 \textsl{Linear connections for systems of second-order ordinary differential equations}\sep
 Ann.\ Inst.\ H.\ Poincar\'e\ \textbf{65} (1996) 223--249.

\bibitem{Crampin.connection.non-autonomous}
  \textsc{Crampin M, Prince GE and Thompson G}\sep
  \textsl{A geometrical version of the Helmholtz conditions in time-dependent
         Lagrangian dynamics}\sep
  J. Phys. A \textbf{17} (1984) 1437--1447.

\bibitem{Douglas}
 \textsc{Crampin M, Sarlet W, Mart\'{\i}nez E, Byrnes G. and Prince GE}\sep
 \textsl{Towards a geometrical understanding of Douglas's solution of the inverse problem of the calculus of variations}\sep
 Inverse Problems\ \textbf{10} (1994) 245--260.

\bibitem{DouMi2}
 \textsc{Doupovec and Mikoulski W}\sep 
 \textsl{Prolongations of pairs of connections into connections on vertical bundles}\sep 
 Arch. Math. (Brno) \textbf{41} (2005) 409--422.

\bibitem{DouMi1}
 \textsc{Doupovec and Mikoulski W}\sep 
 \textsl{On the existence of prolongation of connections}\sep 
 Czechoslovak Math. J. \textbf{56} (2006) 1323--1334.

\bibitem{DouMi3}
 \textsc{Doupovec and Mikoulski W}\sep 
 \textsl{Gauge natural prolongation of connections}\sep 
 Ann. Polon. Math. \textbf{95} (2009) 37--50.

\bibitem{Grifone}
 \textsc{Grifone J}\sep
 \textsl{Structure presque tangente et connexions I}\sep
 Ann. Inst. Fourier \textbf{22} no. 1 (1972) 287--334.

\bibitem{NatOp}
 \textsc{Kol\'a\v{r} I, Michor PW and Slovak J}\sep 
 \textsl{Natural Operations in Differential Geometry}\sep 
 Springer-Verlag, 1993.

\bibitem{IP.nonautonomous}
 \textsc{Krupkov{\'a} O and Prince G}\sep 
 \textsl{Second Order Ordinary Differential Equations in Jet Bundles and the Inverse Problem of the Calculus of Variations}\sep 
 In Handbook of Global Analysis, edited by D. Krupka and D. Saunders, Elsevier (2007) pp 837--904.  

\bibitem{connections.Sarda}
 \textsc{Mangiarotti L and Sardanashvily G}\sep
 \textsl{Connections in Classical and Quantum Field Theory}\sep
 World Scientific Pub. Co. Inc., 2000.

\bibitem{connections.Marsden}
 \textsc{Marsden JE}\sep
 \textsl{Lectures on Mechanics}\sep
 Cambridge University Press, 1992.

\bibitem{MikeFest}
 \textsc{Mart\'{\i}nez E}\sep
 \textsl{Parallel transport and decoupling}\sep
    In the book \textsl{Colloquium on Applied Differential Geometry and Mechanics. (In honour to Mike Crampin, on the occasion of his 60th birthday.)}

\bibitem{tesis}
 \textsc{Mart\'{\i}nez E}\sep
 \textsl{Geometr\'{\i}a de Ecuaciones Diferenciales Aplicada a la Mec\'{a}nica}\sep
 Ph. D. Thesis, University of Zaragoza, Spain; Publicaciones del Seminario Garc\'{\i}a Galdeano \textbf{36}, 1991.
%
\bibitem{jetoids.variational}
 \textsc{Mart\'{\i}nez E}\sep
 \textsl{Classical Field Theory on Lie algebroids: variational aspects}\sep
 J. Phys. A: Math. Gen. \textbf{38} (2005) 7145--7160.

\bibitem{jetoids.multisymplectic}
 \textsc{Mart\'{\i}nez E}\sep
 \textsl{Classical Field Theory on Lie algebroids: multisymplectic formalism}\sep
 Preprint, arXiv:math/0411352 [math.DG].

\bibitem{linearizable.SODE}
  \textsc{Mart{\'\i}nez E and  Cari{\~n}ena JF}\sep
  \textsl{Geometric characterization of linearizable second-order differential equations}\sep
  Math. Proc. Camb. Phil. Soc. \textbf{119} (1996) 373--381.

\bibitem{deriv}
 \textsc{Mart\'{\i}nez E, Cari\~{n}ena JF and Sarlet W}\sep
 \textsl{Derivations of differential forms along the tangent bundle projection}\sep
 Diff.\ Geom.\ Appl. \textbf{2} (1992) 17--43.

\bibitem{deriv2}
 \textsc{Mart\'{\i}nez E, Cari\~{n}ena JF and Sarlet W}\sep
 \textsl{Derivations of differential forms along the tangent bundle projection II}\sep
 Diff.\ Geom.\ Appl. \textbf{3} (1993) 1--29.

\bibitem{decoupling}
  \textsc{Mart{\'\i}nez E, Cari{\~n}ena JF and Sarlet W}\sep
  \textsl{Geometric characterization of separable second-order equations}\sep
  Math. Proc. Camb. Phil. Soc. \textbf{113} (1993) 205--224.

\bibitem{AffineInVector}
 \textsc{Mart\'\i nez E, Mestdag T and Sarlet W}\sep
 \textsl{Lie algebroid structures and Lagrangian systems on affine bundles}\sep
 J.\ Geom.\ Phys. \textbf{44 1}(2002) 70--95.

\bibitem{Massa.Pagani}
 \textsc{Massa E and Pagani E}\sep
 \textsl{Jet bundle geometry, dynamical connections, and the inverse problem of Lagrangian Mechanics}\sep
 Ann. Inst. H. Poincar\'e Phys. Th\'eor. \textbf{61} (1994) 17-62.

\bibitem{con.affine.MeSa}
  \textsc{Mestdag T and Sarlet W}\sep
  \textsl{The Berwald-type connection associated to time-dependent second-order differential equations}\sep
  Houston J. Math. \textbf{27} (2001), no. 4, 763--797.

\bibitem{con.affine.MeSaMa} 
 \textsc{Mestdag T, Sarlet W and Mart\'{\i}nez E}\sep
 \textsl{Note on generalised connections and affine bundles}\sep
 J. Phys. A: Math. Gen. \textbf{35} (2002) 9843--9856.

\bibitem{Poor}
 \textsc{Poor WA}\sep
 \textsl{Differential Geometric Structures}\sep
 McGaw-Hill, New York, 1981.

\bibitem{IP.autonomo}
   \textsc{Sarlet W, Crampin M and Mart{\'\i}nez E}\sep
   \textsl{The integrability conditions in the inverse problem of the calculus of variations for second-order ordinary differential equations}\sep
   Acta Appl. Math. \textbf{54} (1998), no. 3, 233--273. 

\bibitem{generalized.submersive}
 \textsc{Sarlet W, Prince G and Crampin M}\sep
 \textsl{Generalized submersiveness of second-order ordinary differential equations}\sep
 J. Geom. Mech. \textbf{1} (2009) 209--221

\bibitem{Vilms}
 \textsc{Vilms J}\sep 
 \textsl{Connections on tangent bundles}\sep 
 J.\ Diff.\ Geom.\ \textbf{1} (1967) 235--243.

\end{thebibliography}
\end{document}